\documentclass[a4paper,11pt]{smfart}
\usepackage[francais]{babel}
\usepackage[dvips]{graphicx}
\usepackage{amssymb}
\usepackage{amsmath}
\usepackage[utf8]{inputenc}
\usepackage[T1]{fontenc}
\usepackage[francais]{babel}
\usepackage{graphicx}
\usepackage{amsmath}
\usepackage{amsthm}
\usepackage{amsfonts}
\usepackage{amssymb}
\usepackage{enumerate}
\usepackage{vaucanson-g}
\usepackage{graphicx}
\newtheorem{theo}{Th\'eor\`eme}[section]
\newtheorem{prop}[theo]{Proposition}
\newtheorem{lem}[theo]{Lemme}
\newtheorem{coro}[theo]{Corollaire} 

\theoremstyle{definition}

\theoremstyle{remark} 
\newtheorem{rem}[theo]{Remarque}
\theoremstyle{definition}  

\numberwithin{equation}{section}

\newcommand{\Q}{\overline{\mathbb Q}}
\newcommand{\G}{{\bf G}}

\newcommand{\N}{\mathbb{N}}

\newcommand{\lambd}{{\boldsymbol{\lambda}}}
\newcommand{\f}{{\bf f}}
\renewcommand{\O}{\mathcal{O}}

\renewcommand{\k}{{\bf{k}}}

\renewcommand{\v}{{\bf v }}
\newcommand{\ev}{{\rm ev}}

\newcommand{\w}{{\bf{w}}}

\newcommand{\g}{{\bf g}}

\title[]{M\'ethode de Mahler : relations lin\'eaires, transcendance et applications aux nombres automatiques}

\author{Boris Adamczewski} 
\address{
Univ Lyon, Universit\'e Claude Bernard Lyon 1\\
CNRS UMR 5208, Institut Camille Jordan  \\
43 blvd du 11 novembre 1918 \\
F-69622 Villeurbanne Cedex, France}
\email{Boris.Adamczewski@math.cnrs.fr}

\author{Colin Faverjon}
\email{colin.faverjon@ac-creteil.fr}

\begin{abstract} Cet article est consacr\'e \`a la m\'ethode de Mahler. Nous d\'ecrivons en d\'etail la structure des relations de d\'ependance 
lin\'eaire entre les valeurs aux points alg\'ebriques de fonctions mahl\'eriennes.  
\'Etant donn\'es un corps de nombres ${\bf k}$, une fonction mahl\'erienne $f(z)\in{\bf k}\{z\}$ et $\alpha$ un nombre alg\'ebrique, $0<\vert \alpha\vert <1$, 
qui n'est pas un p\^ole de 
$f$, nous montrons notamment que l'on peut toujours d\'eterminer si le nombre $f(\alpha)$ est transcendant ou non. 
Dans ce dernier cas, nous obtenons que $f(\alpha)$ appartient n\'ecessairement \`a l'extension ${\bf k}(\alpha)$. 
Nous consid\'erons \'egalement  les cons\'equences  de cette th\'eorie  
concernant un probl\`eme arithm\'etique classique :  l'\'etude de la suite des chiffres des nombres alg\'ebriques dans une base enti\`ere ou, 
plus g\'en\'eralement, alg\'ebrique. Nos r\'esultats sont obtenus \`a partir d'un th\'eor\`eme r\'ecent de Philippon \cite{PPH} que nous raffinons et 
dont nous simplifions la d\'emonstration. 
\end{abstract}

\thanks{This project has received funding from the European Research Council (ERC) under the European Union's Horizon 2020 research and innovation programme 
under the Grant Agreement No 648132. }

\begin{document}
\maketitle

\tableofcontents
\section{Introduction}

\'Etant donn\'e un entier $q\geq 2$, une fonction  $f(z)\in \overline{\mathbb Q}\{z\}$ est  dite {\it $q$-mahl\'erienne}  
s'il existe des polyn\^omes $p_0(z),\ldots , p_n(z)\in \overline{\mathbb Q}[z]$, non tous nul, 
tels que 
\begin{equation} \label{eq}
 p_0(z)f(z)+p_1(z)f(z^q)+\cdots + p_n(z)f(z^{q^n}) \ = \ 0. 
 \end{equation}  
 Afin d'\'etudier ces fonctions, il est souvent commode de consid\'erer les syst\`emes d'\'equations fonctionnelles de la forme : 

\begin{equation}\label{eq: systeme}
\left( \begin{array}{ c }
     f_1(z) \\
     \vdots \\
     f_n(z)
  \end{array} \right) = A(z)\left( \begin{array}{ c }
     f_1(z^q) \\
     \vdots \\
     f_n(z^q)
  \end{array} \right)  \, ,
\end{equation}
o\`u  $A(z)$ est une matrice de $\rm{GL}_n(\overline{\mathbb Q}(z))$  et les $f_i$ sont des fonctions de la variable 
$z$, analytiques au voisinage de $z=0$.  Un tel syst\`eme est appel\'e mahl\'erien. Par abus de langage, nous dirons que l'entier $n$ est l'ordre du syst\`eme (\ref{eq: systeme}). 
Une fonction $f(z)\in \overline{\mathbb Q}\{z\}$ est $q$-mahl\'erienne si, 
et seulement si, $f(z)$ est une coordonn\'ee 
d'un vecteur solution d'un syst\`eme mahl\'erien.   

Un nombre complexe $\alpha$ du disque unit\'e ouvert est une {\it singularit\'e} du syst\`eme (\ref{eq: systeme}) s'il existe un entier positif $\ell$ tel que 
$\alpha^{q^l}$ soit un p\^ole d'un coefficient d'une des matrices $A(z)$ ou $A(z)^{-1}$. 
L'ensemble des singularit\'es d'un syst\`eme mahl\'erien peut donc \^etre infini, 
mais il ne contient aucun point d'accumulation \`a l'int\'erieur du disque unit\'e. 
Un \'el\'ement du disque unit\'e complexe ouvert qui n'est pas une singularit\'e 
est dit {\it r\'egulier} pour le syst\`eme (\ref{eq: systeme}).

\medskip

La m\'ethode de Mahler, introduite \`a la fin des ann\'ees vingt,  vise  \`a prouver des r\'esultats de transcendance et d'ind\'ependance 
alg\'ebrique pour les valeurs aux points 
alg\'ebriques r\'eguliers  de telles fonctions.  
Pour les aspects classiques de la th\'eorie, nous renvoyons le lecteur \`a la monographie de Ku. Nishioka \cite{Ni_Liv}.   
\'Evidemment, les d\'efinitions pr\'ec\'edentes t\'emoignent d'une forte analogie avec les $E$-fonctions introduites par Siegel : 
les \'equations diff\'erentielles sont remplac\'ees par  des \'equations aux diff\'erences associ\'ees \`a l'endomorphisme injectif de 
$\mathbb C[[z]]$ d\'efini par $\sigma_q(f)=f(z^q)$.  Notons toutefois deux diff\'erences importantes. 
Tout d'abord, une fonction mahl\'erienne n'est pas une fonction enti\`ere (sauf si c'est un polyn\^ome), mais une fonction m\'eromorphe sur le disque unit\'e ouvert, 
le cercle unit\'e formant une fronti\`ere naturelle \cite{Ra92}. 
Ensuite, une fonction mahl\'erienne non nulle peut prendre des valeurs alg\'ebriques en un nombre infini de points alg\'ebriques du disque unit\'e ouvert. 
La th\'eorie pr\'esente toutefois dans son d\'eveloppement une forte analogie avec celle des $E$-fonctions. 
L'analogue du th\'eor\`eme de Siegel--Shidlovskii a finalement \'et\'e obtenu par Ku. Nishioka en 1990, apr\`es plusieurs r\'esultats partiels de diff\'erents auteurs dont 
Mahler, Kubota, Loxton et van der Poorten. 

\begin{theo}[Nishioka]\label{thm: nishioka}
Soient $f_1(z),\ldots,f_n(z)\in \Q\{z\}$  des fonctions analytiques convergentes sur le disque ouvert de rayon $\rho>0$  
et solutions d'un syst\`eme du type (\ref{eq: systeme}).   
Soit $\alpha\in\Q$, $0<\vert \alpha\vert <\rho$,  un point r\'egulier pour ce syst\`eme.    
Alors
$$
\mbox{\rm degtr}_{\Q} (f_1(\alpha),\ldots,f_n(\alpha)) = \mbox{\rm degtr}_{\Q(z)} (f_1(z),\ldots,f_n(z))\, .
$$
\end{theo}

\medskip

En 2006, Beukers \cite{Beu06} a obtenu, comme cons\'equence de travaux d'Andr\'e \cite{An1,An2}, une version 
raffin\'ee du th\'eor\`eme de Siegel--Shidlovskii.   
Ce r\'esultat remarquable stipule que toute relation alg\'ebrique sur $\Q$ entre les valeurs de $E$-fonctions solutions d'un m\^eme syst\`eme 
 diff\'erentiel,  en un point alg\'ebrique regulier pour ce syst\`eme, s'obtient comme  sp\'ecialisation en ce point d'une relation alg\'ebrique sur $\Q(z)$ entre ces $E$-fonctions.  
Une autre d\'emonstration de ce r\'esultat a r\'ecemment \'et\'e donn\'ee par Andr\'e dans \cite{An3}.  
Inspir\'e par ces travaux, ainsi que par ceux de  Nesterenko et Shidlovskii \cite{NS96}, 
Philippon \cite{PPH} a montr\'e comment  on peut d\'eduire un raffinement similaire, dans le contexte des syt\`emes mahl\'eriens, \`a partir du  th\'eor\`eme \ref{thm: nishioka}.   

\begin{theo}[Philippon]\label{thm: pph}
Soient $f_1(z),\ldots,f_n(z)\in \Q\{z\}$  des fonctions solutions d'un syst\`eme du type (\ref{eq: systeme}).   
Soit $\alpha\in\Q$, $0<\vert \alpha\vert <1$,  un point r\'egulier pour ce syst\`eme.    
Alors, pour tout  $P\in \Q[X_1,\ldots,X_n]$, de degr\'e total $d$, tel que $P(f_1(\alpha),\ldots,f_n(\alpha))=0$, 
il existe $Q\in \Q(z)[X_1,\ldots,X_n]$, de degr\'e total $d$ en $X_1,\ldots, X_n$, tel que 
$Q(z,f_1(z),\ldots,f_n(z))=0$ et $Q(\alpha,X_1,\ldots,X_n)=P(X_1,\ldots,X_n)$. 
\end{theo}

\begin{rem} 
Lorsque Philippon nous a indiqu\'e avoir d\'emontr\'e le th\'eor\`eme \ref{thm: pph} en janvier 2015, nous lui avons indiqu\'e  les applications que 
nous savions en tirer, \`a savoir le th\'eor\`eme \ref{thm: baker} et ses cons\'equences pour les nombres automatiques. 
Ce sont ces applications que nous pr\'esentons ici, ainsi que certains raffinements plus r\'ecents.  
Apr\`es cette discussion, Philippon a trouv\'e une autre approche lui permettant d'obtenir une version affaiblie du th\'eor\`eme 
 \ref{thm: baker}, laquelle est devenue le th\'eor\`eme 1.5 de \cite{PPH} (voir \'egalement la discussion \cite[p.\.4]{PPH}). 
\end{rem}

Dans cet article, nous montrons tout d'abord comment simplifier la d\'emonstration du th\'eor\`eme \ref{thm: pph} et en obtenir  
une version homog\`ene. Ce raffinement est vraiment l'exact analogue 
du th\'eor\`eme principal de Beukers dans \cite{Beu06}. Comme nous le verrons par la suite, disposer d'un \'enonc\'e homog\`ene 
s'av\`ere tr\`es utile pour l'\'etude des relations lin\'eaires. 

\begin{theo}
\label{thm: pphHomogene}
Soient $f_1(z),\ldots,f_n(z)\in \Q\{z\}$  des fonctions solutions d'un syst\`eme du type (\ref{eq: systeme}). 
Soit $\alpha\in\Q$, $0<\vert \alpha\vert <1$,  un point r\'egulier pour ce syst\`eme.  
Alors, pour tout  polyn\^ome homog\`ene $P\in \Q[X_1,\ldots,X_n]$ tel que 
$P(f_1(\alpha),\ldots,f_n(\alpha))=0$, 
il existe un polyn\^ome $Q \in {\Q}[z,X_1,\ldots,X_n]$, homog\`ene en $X_1,\ldots,X_n$, tel que
$Q(z,f_1(z),\ldots,f_n(z))=0$ et  $Q(\alpha,X_1,\ldots,X_n)=P(X_1,\ldots,X_n)$. 
\end{theo}

Le th\'eor\`eme \ref{thm: pphHomogene} implique le th\'eor\`eme \ref{thm: pph}.  
En effet, on peut toujours  transformer une relation inhomog\`ene en une relation homog\`ene en ajoutant au syst\`eme la fonction 
$f_{n+1}$ constante et \'egale \`a $1$. Dans ce nouveau syst\`eme, la matrice $A(z)$ est remplac\'ee par la matrice 
$$\left(\begin{array}{c|c} \begin{array}{cccc} 
\\ &  A(z)& &
\\ \\
 \end{array} & 0 
\\ \hline
\\0  & \begin{array}{ccc}1 
 \end{array}
\end{array}\right)$$
et l'ensemble des point r\'eguliers reste inchang\'e.   

Soit ${\bf k}$ un sous-corps de $\mathbb C$. Soit $\alpha$ un point du disque unit\'e complexe ouvert tel que les fonctions 
$f_1(z),\ldots,f_n(z)$ soient toutes d\'efinies en $\alpha$, 
c'est-\`a-dire, tel que $\alpha$ n'est p\^{o}le d'aucune de ces fonctions.  On d\'efinit le 
${\bf k}$-espace vectoriel des relations lin\'eaires entre les valeurs des fonctions $f_i$ au point $\alpha$ par :
$$
{\rm Rel}_{\bf k}(f_1(\alpha),\ldots,f_n(\alpha)) := \left\{ (\lambda_1,\ldots,\lambda_n) \in {\bf k}^n \mid  \sum_{i=1}^n\lambda_if_i(\alpha) =0 \right\} \, .
$$
On d\'efinit \'egalement le ${\bf k}(z)$-espace vectoriel des relations lin\'eaires fonctionnelles entre les $f_i(z)$ par :
$$
{\rm Rel}_{{\bf k}(z)}(f_1(z),\ldots,f_n(z)) := \left\{ (w_1(z),\ldots,w_n(z)) \in {\bf k}(z)^n \mid  \sum_{i=1}^nw_i(z)f_i(z) = 0 \right\} \,.
$$
Enfin, on note $\ev_\alpha$ l'application d'\'evaluation en $z=\alpha$. 
Dans le cas d'un polyn\^ome homog\`ene de degr\'e un, le th\'eor\`eme \ref{thm: pphHomogene} s'\'enonce alors de la fa\c con suivante.

\begin{coro}
\label{coro: pphLineaire}
Soit $\alpha\in\Q$, $0<\vert \alpha\vert <1$,  un point r\'egulier pour le syst\`eme~(\ref{eq: systeme}). On a : 
$$
{\rm Rel}_{\Q}(f_1(\alpha),\ldots,f_n(\alpha)) = \ev_\alpha \left({\rm Rel}_{\Q(z)}(f_1(z),\ldots,f_n(z)) \right) \, .
$$
En particulier, si les fonctions $f_1(z),\ldots,f_n(z)$ sont lin\'eairement ind\'ependantes sur $\Q (z)$, 
alors les nombres $f_1(\alpha),\ldots,f_n(\alpha)$ sont lin\'eairement ind\'ependants sur $\Q$.
\end{coro}

\begin{rem} Obtenir l'ind\'ependance lin\'eaire sur $\Q$ des nombres $f_1(\alpha),\ldots,f_n(\alpha)$ \`a partir du th\'eor\`eme \ref{thm: pph}, 
n\'ecessite une condition plus forte : l'ind\'ependance lin\'eaire sur  $\Q(z)$ des fonctions $1,f_1(z),\ldots,f_n(z)$ (ou bien que l'une des $f_i$ soit constante).  
Une telle condition n'est en g\'en\'eral pas v\'erifi\'ee.  Par exemple, \`a tout automate fini on peut associer un 
syst\`emes mahl\'erien pour lequel la somme des fonctions $f_i(z)$ est \'egale \`a $1/(1-z)$.  
Le th\'eor\`eme 1.7 de \cite{PPH} permet toutefois d'obtenir une telle conclusion dans un cas  tr\`es particulier : le syst\`eme doit admettre 
une matrice fondamentale de solutions dont les coefficients sont des fonctions analytiques dans le disque unit\'e ouvert.  
Contrairement \`a ce qui est affirm\'e dans \cite{PPH}, le corollaire \ref{coro: pphLineaire} montre qu'une telle restriction  
n'est pas n\'ecessaire.  
\end{rem}

Les th\'eor\`emes \ref{thm: pph} et \ref{thm: pphHomogene} permettent en fait d'obtenir des r\'esultats 
valables \'egalement en des points singuliers. 
Ainsi, un aspect remarquable  
du th\'eor\`eme suivant est  qu'aucune condition de r\'egularit\'e n'est requise pour le point $\alpha$. 
Nous donnerons en outre  deux d\'emonstrations diff\'erentes du point (i).

\begin{theo}\label{thm: baker}
Soient $f_1(z),\ldots,f_n(z)$ des fonctions $q$-mahl\'eriennes.   
Soient $\alpha\in\Q$, $0<\vert \alpha\vert <1$,  un nombre qui n'est p\^ole d'aucune de ces fonctions et ${\bf k}$ un corps de nombres contenant 
$\alpha$ ainsi que les coefficients des $f_i$.  
\begin{itemize}

\item[{\rm (i)}]Si les nombres $f_1(\alpha), \ldots,f_n(\alpha)$ sont lin\'eairement d\'ependants sur $\Q$, alors ils sont lin\'eairement d\'ependants sur 
${\bf k}$. 

\item[{\rm (ii)}] Plus pr\'ecis\'ement, on a  :
\begin{equation*}
{\rm Rel}_{\Q}\left(f_1(\alpha),\ldots,f_n(\alpha)\right) = {\rm Vect}_{\Q} \left\{ {\rm Rel}_{{\bf k}}\left(f_1(\alpha),\ldots,f_n(\alpha)\right)\right\}.
\end{equation*}
\end{itemize}
\end{theo}

En utilisant le fait que la fonction $g\equiv 1$ est $q$-mahl\'erienne pour tout $q\geq 2$, on obtient imm\'ediatement le r\'esultat suivant.

\begin{coro}\label{cor: alt}
Soient $f(z)$ une fonction $q$-mahlerienne et $\alpha\in {\Q}$, $0<\vert \alpha\vert<1$, qui n'est pas un p\^ole de $f$. Soit ${\bf k}$ un corps de 
nombres contenant $\alpha$ ainsi que les coefficients de $f$.  
On a l'alternative suivante : soit $f(\alpha)$ 
est transcendant, soit $f(\alpha)\in{\bf k}$. 
\end{coro}

Dans le cas particulier o\`u $f(z)$ est une s\'erie automatique et $\alpha$ est un nombre rationnel, ce r\'esultat a \'et\'e conjectur\'e par Cobham en 1968 \cite{Cob68}. 
C'est pr\'ecis\'ement cette conjecture qui est \`a l'origine du pr\'esent travail. Nous donnons davantage 
de d\'etails concernant l'histoire de ce probl\`eme dans l'appendice A.   
Le corollaire \ref{cor: alt} semble \^etre le premier r\'esultat de transcendance compl\`etement g\'en\'eral obtenu par la m\'ethode de Mahler 
(i.e.\ valable pour toute fonction 
mahl\'erienne et en tout point alg\'ebrique de son domaine de d\'efinition).  
D'autre part, des exemples montrent que l'on ne peut se soustraire \`a l'alternative pr\'esente dans la conclusion du corollaire \ref{cor: alt}, 
m\^eme en supposant la fonction $f(z)$ transcendante (voir section \ref{sec: ex}).  

\medskip

Nous pr\'ecisons ensuite le corollaire \ref{coro: pphLineaire} et d\'ecrivons en d\'etail l'espace vectoriel ${\rm Rel}_{{\bf k}}\left(f_1(\alpha),\ldots,f_n(\alpha)\right)$,  
m\^eme lorsque $\alpha$ est une singularit\'e du syst\`eme. 
\'Etant donn\'es un syst\`eme du type (\ref{eq: systeme}) et un entier $\ell\geq 1$, on pose 
$$
A_l(z) := A(z)A(z^q)\cdots A(z^{q^{l-1}})
$$ 
et  
$$
{\rm ker}_{\bf k} A_l(\alpha) := \left\{ (\lambda_1,\ldots,\lambda_n) \in {\bf k}^n\mid  (\lambda_1,\ldots,\lambda_n) A_l(\alpha) = 0\right\} \, .
$$
On fixe \'egalement un nombre r\'eel 
$\rho$, $0<\rho<1$, strictement  inf\'erieur au minimum des modules des p\^oles (non nuls) de la matrice $A(z)$ et des racines (non nulles) de son d\'eterminant. 
Ainsi, la matrice $A(z)$ est d\'efinie et inversible sur le disque \'epoint\'e $D(0,\rho)^\star$.  

\begin{theo}
\label{thm: StructureRelationsLineaires}
Soient $f_1(z),\ldots,f_n(z)$  des fonctions solutions d'un syst\`eme du type (\ref{eq: systeme}).  Soient $\alpha\in\Q$, $0<\vert \alpha\vert <1$,  
et ${\bf k}$ un corps de nombres  contenant 
$\alpha$ ainsi que les coefficients des $f_i$.
Soit $l$ un entier tel que $|\alpha^{q^l} | < \rho $. 
Si $\alpha$ n'est pas un p\^{o}le de $A_l(z)$, 
alors
\begin{equation*}
\label{StructureRelationsLineaires}
{\rm Rel}_{\bf k}(f_1(\alpha),\ldots,f_n(\alpha)) = {\rm ker}_{\bf k} A_l(\alpha) + {\rm ev}_\alpha \left({\rm Rel}_{{\bf k}(z)}(f_1(z),\ldots,f_n(z))\right) \,.
\end{equation*}
\end{theo}

La seule restriction dans le th\'eor\`eme pr\'ec\'edent  vient du fait que l'on doit supposer que le nombre $\alpha$ n'est pas un p\^{o}le de $A_l(z)$. 
En compl\'ement, nous montrons le r\'esultat suivant. 

\begin{theo}
\label{thm: paspole}
Soient $f_1(z),\ldots,f_n(z)$  des fonctions solutions d'un syst\`eme du type (\ref{eq: systeme}).  Soient $\alpha\in\Q$, $0<\vert \alpha\vert <1$,  
et ${\bf k}$ un corps de nombres  contenant 
$\alpha$ ainsi que les coefficients des $f_i$. Supposons que les $f_i$ soient d\'efinies au point $\alpha$. Soit  $l$ 
un entier tel que $|\alpha^{q^l} | < \rho $.   
Alors, il existe une matrice $B(z)\in{\rm GL}_n({\bf k}(z))$ satisfaisant aux conditions suivantes. 

\medskip

\begin{itemize}

\item[{\rm (i)}] On a : \begin{equation*}\label{eq: systeme2}
\left( \begin{array}{ c }
     f_1(z) \\
     \vdots \\
     f_n(z)
  \end{array} \right) = B(z)\left( \begin{array}{ c }
     f_1(z^{q^l}) \\
     \vdots \\
     f_n(z^{q^l})
  \end{array} \right)  \, .
\end{equation*}

\medskip

\item[{\rm (ii)}] Le point $\alpha$ n'est pas un p\^ole de $B(z)$. 

\medskip

\item[{\rm (iii)}] Le point $\alpha^{q{^l}}$ est r\'egulier pour le syst\`eme ${\rm (i)}$.

\end{itemize}

\medskip

En outre, si les fonctions $f_1(z),\ldots,f_n(z)$ sont lin\'eairement ind\'ependantes sur ${\bf k}(z)$, on a  
$B(z)=A_l(z)$. 
\end{theo}


Les th\'eor\`emes \ref{thm: baker} et \ref{thm: StructureRelationsLineaires} d\'ecrivent totalement la structure des relations lin\'eaires 
entre les valeurs de fonctions solutions d'un syst\`eme mahl\'erien en un point alg\'ebrique $\alpha$ de leur domaine d'holomorphie. 
Il en existe de deux sortes : les relations d'origine \og matricielle\fg{} et celles d'origine \og fonctionnelle\fg. 
Les relations matricielles sont les \'el\'ements de l'espace ${\rm ker}_{\Q} A_l(\alpha)$ et leur recherche se r\'eduit donc au calcul du noyau  
d'une matrice explicite. Les relations d'origine fonctionnelle correspondent aux \'el\'ements de l'espace 
${\rm ev}_\alpha \left({\rm Rel}_{{\Q}(z)}(f_1(z),\ldots,f_n(z))\right)$.  Pour les trouver, il faut donc \^etre capable de d\'eterminer une base de l'espace 
${\rm Rel}_{{\Q}(z)}(f_1(z),\ldots,f_n(z))$ des relations de d\'ependance lin\'eaire entre 
fonctions d'un syst\`eme mahl\'erien.  Rappelons que d\'ecrire les relations de d\'ependance alg\'ebrique entre les solutions d'un syst\`eme mahlerien 
est une t\^ache ardue, en t\'emoigne le peu de r\'esultats obtenus jusqu'\`a pr\'esent (voir par exemple \cite[Chap.\ 5]{Ni_Liv} et plus r\'ecemment \cite{BCZ,Ro}). 
\textit{A contrario}, nous montrerons que l'espace ${\rm Rel}_{{\Q}(z)}(f_1(z),\ldots,f_n(z))$ a une description simple.  
Il est engendr\'e par des relations lin\'eaires de \og petits degr\'es\fg, calculables de mani\`ere effective.
Plus pr\'ecis\'ement, le th\'eor\`eme \ref{BaseRelationsLin\'eaires1} montre que ${\rm Rel}_{{\Q}(z)}(f_1(z),\ldots,f_n(z))$ est isomorphe au 
noyau d'une matrice explicite. 


Finalement, les th\'eor\`emes  \ref{thm: StructureRelationsLineaires}, \ref{thm: paspole}, \ref{BaseRelationsLin\'eaires1} et \ref{BaseRelationsLin\'eaires2}, 
et leurs d\'emonstrations fournissent un algorithme permettant de r\'epondre en toute g\'en\'eralit\'e \`a la question suivante : 
\'etant donn\'es $f(z)$ une fonction mahl\'erienne\footnote{De fa\c con \'equivalente, $f(z)$ peut-\^etre 
donn\'ee par un syst\`eme du type \ref{eq: systeme} ou une \'equation du type \ref{eq}. } et $\alpha$, $0< \vert \alpha\vert <1$, un nombre alg\'ebrique,  $f(\alpha)$ est-il  
alg\'ebrique ou transcendant ?  Plus g\'en\'eralement, \'etant donn\'es des fonctions $q$-mahl\'eriennes $f_1(z),\ldots,f_n(z)$ et  $\alpha$, $0< \vert \alpha\vert <1$, un 
nombre alg\'ebrique qui n'est p\^ole d'aucune de ces fonctions, on peut d\'eterminer de fa\c con algorithmique une base de l'espace vectoriel 
${\rm Rel}_{\overline{\mathbb Q}}(f_1(\alpha),\ldots,f_n(\alpha))$.

\medskip

Cet article est organis\'e comme suit. Dans la section \ref{sec: rem}, nous revenons sur la d\'emonstration du th\'eor\`eme de Philippon, puis nous en d\'emontrons 
une version homog\`ene (le th\'eor\`eme \ref{thm: pphHomogene}) dans la section \ref{sec: pphhom}. Une premi\`ere d\'emonstration du point (i) 
du th\'eor\`eme \ref{thm: baker} est donn\'ee dans la section \ref{sec: pointi}, tandis que les th\'eor\`emes  \ref{thm: baker}, \ref{thm: StructureRelationsLineaires}  
et \ref{thm: paspole} sont d\'emontr\'es dans la section \ref{sec: preuves}. Dans la section \ref{sec: relfonc}, nous \'etudions les relations fonctionnelles 
de d\'ependance lin\'eaire entre les solutions d'un syst\`eme mahl\'erien et d\'emontrons les th\'eor\`emes \ref{BaseRelationsLin\'eaires1}  et \ref{BaseRelationsLin\'eaires2}. 
Dans la section \ref{sec: ex}, nous illustrons les   r\'esultats obtenus \`a travers l'\'etude de deux exemples de syst\`emes mahl\'eriens automatiques.  
Enfin, dans un appendice, nous donnons des \'el\'ements historiques 
concernant la conjecture de Cobham sur la suite des chiffres des nombres alg\'ebriques dans une base enti\`ere, ainsi que ses liens avec la m\'ethode de Mahler et 
la th\'eorie des automates finis. Cette conjecture,  
propos\'ee en 1968 et qui d\'ecoule du th\'eor\`eme \ref{thm: baker}, a \'et\'e la source principale de motivation pour le pr\'esent travail.    

\section{Remarques sur la d\'emonstration de Philippon}\label{sec: rem}

Nous revenons tout d'abord sur la d\'emonstration du th\'eor\`eme \ref{thm: pph} 
donn\'ee par Philippon dans \cite{PPH}.  Notre but est de montrer comment en simplifier l'exposition. 
La d\'emonstration de Philippon se d\'ecompose en deux parties principales. 
Tout d'abord, une premi\`ere \'etape consiste \`a montrer que le th\'eor\`eme est vrai pour tout point 
$\alpha$ appartenant \`a un certain voisinage  de l'origine. 
C'est ce que nous appellerons ici le \og{}th\'eor\`eme local\fg{} et qui correspond \`a la proposition 4.4 de \cite{PPH}. 
La seconde \'etape  est assez courte et correspond \`a la d\'emonstration du corollaire 4.5 dans \cite{PPH} ; il s'agit de  
montrer que le th\'eor\`eme local implique en fait le th\'eor\`eme global. L'id\'ee astucieuse introduite par Philippon est la suivante. 
Si $\alpha$ n'est pas une singularit\'e du syst\`eme, une relation alg\'ebrique entre les fonctions $f_i$ au point $\alpha$ se transporte naturellement, 
par it\'eration de l'\'equation fonctionnelle, en une relation alg\'ebrique aux points $\alpha^{q^l}$. Pour $l$ suffisamment grand, $\alpha^{q^l}$ 
appartient au domaine de validit\'e du th\'eor\`eme local, que 
l'on peut donc appliquer. Le fait que $\alpha$ ne soit pas une singularit\'e permet finalement d'obtenir, par it\'eration de la matrice inverse 
qui est bien d\'efinie,  le th\'eor\`eme au point $\alpha$.

Nous nous int\'eressons maintenant \`a la d\'emonstration du th\'eor\`eme local. 
Philippon suit la d\'emarche introduite par Nesterenko et Shidlovskii \cite{NS96} dans le cadre des $E$-fonctions. 
Nous nous proposons d'axiomatiser un peu celle-ci en extrayant de \cite{NS96} le r\'esultat suivant qui ne requiert 
la pr\'esence d'aucune \'equation diff\'erentielle ou fonctionnelle. La d\'emonstration de la proposition \ref{prop: NS} donn\'ee ici reprend les arguments de 
\cite{NS96} et \cite{PPH}. Il s'agit simplement d'une \'egalit\'e de dimension qui repose sur des principes de base d'alg\`ebre commutative et un r\'esultat de Krull \cite{Kr}. 
Elle est \'egalement \`a rapprocher du Corollary 1.7.1 obtenu par Andr\'e dans \cite{An3}. 

\begin{prop}\label{prop: NS}
Soient $f_1(z),\ldots, f_n(z)\in \Q\{z\}$. Supposons que les hypoth\`eses suivantes sont v\'erifi\'ees. 

\begin{itemize}

\medskip

\item[{\rm (i)}] Il existe $\rho>0$ tel que pour tout nombre alg\'ebrique $\alpha$, $0<\vert \alpha\vert<\rho$, on ait : 
$$
\mbox{\rm degtr}_{\Q} (f_1(\alpha),\ldots,f_n(\alpha)) =\mbox{\rm degtr}_{\Q(z)} (f_1(z),\ldots,f_n(z))\, .
$$
\medskip

\item[{\rm (ii)}] L'extension $\Q(z)(f_1(z),\ldots,f_n(z))$ est r\'eguli\`ere, ce qui signifie que tout \'el\'ement de $\Q(z)(f_1(z),\ldots,f_n(z))$ 
alg\'ebrique sur $\Q(z)$ est un \'el\'ement de $\Q(z)$. 
\end{itemize}
\medskip

Alors, il existe $\rho'>0$ tel que pour tout nombre alg\'ebrique $\alpha$, $0<\vert \alpha\vert<\rho'$ et  
tout  polyn\^ome $P\in \Q[X_1,\ldots,X_n]$, de degr\'e total $d$, tel que $P(f_1(\alpha),\ldots,f_n(\alpha))=0$, 
il existe $Q\in \Q(z)[X_1,\ldots,X_n]$, de degr\'e total $d$ en $X_1,\ldots, X_n$, tel que 
$Q(z,f_1(z),\ldots,f_n(z))=0$ et $Q(\alpha,X_1,\ldots,X_n)=P(X_1,\ldots,X_n)$. 
\end{prop}

\begin{proof}
On note $\mathfrak P$ l'id\'eal premier de $\Q(z)[X_1,\ldots,X_n]$ des relations alg\'ebriques sur $\Q(z)$ entre les fonctions $f_1(z),\ldots,f_n(z)$. 
\'Etant donn\'e $\alpha$ un nombre alg\'ebrique tel que les fonctions $f_i$ soient toutes d\'efinies en $\alpha$, on note \'egalement 
$\mathfrak P_{\alpha}$ l'id\'eal premier de $\Q[X_1,\ldots,X_n]$ des relations alg\'ebriques sur $\Q$ entre les nombres $f_1(\alpha),\ldots,f_n(\alpha)$. 
Soit $X_0$ une nouvelle ind\'etermin\'ee. On note $\tilde{\mathfrak P}\subset \Q(z)[X_0,\ldots,X_n]$ (respectivement $\tilde{\mathfrak P}_{\alpha}\subset \Q[X_0,\ldots,X_n]$) 
l'id\'eal 
homog\'en\'eis\'e en $X_1,\ldots,X_n$ de $\mathfrak P$ (respectivement de $\mathfrak P_{\alpha}$).  Rappelons que l'homog\'en\'eis\'e d'un id\'eal premier 
est un id\'eal premier de m\^eme rang.  Le rang d'un id\'eal premier $\mathfrak I$, qui est parfois \'egalement appel\'e la hauteur de $\mathfrak I$, est not\'e ici ${\rm rg}(\mathfrak I)$. 
On note \'egalement $\dim A$, la dimension de Krull d'un anneau commutatif unitaire 
$A$. Comme pr\'ec\'edemment, on note $\mbox{\rm ev}_{\alpha} : \Q[z] \rightarrow \Q$ l'application d'\'evaluation en $z=\alpha$. 

Avec ces notations, on v\'erifie que la conclusion du th\'eor\`eme est \'equivalente \`a l'\'egalit\'e 
\begin{equation}\label{eq: ev}
{\rm ev}_{\alpha}(\tilde{\mathfrak P}\cap \Q[z,X_0,\ldots,X_n])= \tilde{\mathfrak P}_{\alpha} \, ,
\end{equation}
pour tout nombre alg\'ebrique non nul $\alpha$ de module suffisamment petit. 
Comme l'anneau $\Q(z)[f_1(z),\ldots,f_n(z)]$ est de type fini et int\`egre, on a : 
\begin{eqnarray*}
\mbox{\rm degtr}_{\Q(z)} (f_1(z),\ldots,f_n(z))&=& \dim \Q(z)[f_1(z),\ldots,f_n(z)]\\
&=& \dim (\Q(z)[X_1,\ldots,X_n]/\mathfrak P) \\
& = &\dim(\Q(z)[X_1,\ldots,X_n]) - \mbox{\rm rg}(\mathfrak P) \\
 & = &\dim \Q[X_0,\ldots,X_n] - \mbox{\rm rg}(\tilde{\mathfrak P}) - 1.
\end{eqnarray*}
De fa\c con totalement similaire, il vient :
$$
\mbox{\rm degtr}_{\Q} (f_1(\alpha),\ldots,f_n(\alpha)) = \dim \Q[X_0,\ldots,X_n] - \mbox{\rm rg}(\tilde{\mathfrak P}_{\alpha}) - 1 \, .
$$
D'apr\`es (i), on en d\'eduit que pour tout nombre alg\'ebrique $\alpha$ tel que $0<\alpha<\rho$ : 
\begin{equation}\label{eq: rg}
\mbox{\rm rg} (\tilde{\mathfrak P})=\mbox{\rm rg} (\tilde{\mathfrak P}_{\alpha}) \,.
\end{equation} 
Compte tenu de (\ref{eq: rg}) et comme ${\rm ev}_{\alpha}(\tilde{\mathfrak P}\cap \Q[z,X_0,\ldots,X_n])\subset  \tilde{\mathfrak P}_{\alpha}$,  
il suffit de montrer que, pour tout $\alpha$ suffisamment petit, l'id\'eal  
${\rm ev}_{\alpha}(\tilde{\mathfrak P}\cap \Q[z,X_0,\ldots,X_n])$ est un id\'eal premier de m\^eme rang que $\tilde{\mathfrak P}$ 
pour obtenir (\ref{eq: ev}) et conclure. 

D'apr\`es (ii), l'extension $\Q(z)(f_1(z),\ldots,f_n(z))$ est r\'eguli\`ere, ce qui implique que l'id\'eal $\tilde{\mathfrak P}\cap \Q[z,X_0,\ldots,X_n])$ 
est absolument premier (voir \cite[Theorem 39]{ZS}). Un r\'esultat de Krull \cite{Kr} implique alors que l'id\'eal 
${\rm ev}_{\alpha}(\tilde{\mathfrak P}\cap \Q[z,X_0,\ldots,X_n])$ est premier 
pour tout $\alpha$ en dehors d'un ensemble fini. Lorsque l'id\'eal  ${\rm ev}_{\alpha}(\tilde{\mathfrak P}\cap \Q[z,X_0,\ldots,X_n])$ est premier, 
on peut montrer l'\'egalit\'e de rang 
$$
{\rm rg}({\rm ev}_{\alpha}(\tilde{\mathfrak P}\cap \Q[z,X_0,\ldots,X_n])) = {\rm rg}(\tilde{\mathfrak P})
$$
comme dans \cite{NS96} \`a l'aide d'un r\'esultat de Hilbert. Cela conclut la d\'emonstration.  
\end{proof}

Ainsi, pour obtenir le th\'eor\`eme local, il suffit de disposer du th\'eor\`eme \ref{thm: nishioka} de Nishioka et du lemme suivant. 
C'est pour d\'emontrer un r\'esultat analogue au lemme \ref{lem: reg} que Philippon a recours \`a la th\'eorie de Galois aux diff\'erences, notamment 
\`a l'utilisation des propositions 2.2,  3.2, 3.4 
et du lemme 3.5 dans \cite{PPH}. 
Nous donnons ci-dessous une d\'emonstration tr\`es courte de ce r\'esultat qui ne n\'ecessite aucun usage de la th\'eorie de 
Galois aux diff\'erences et simplifie notablement l'exposition donn\'ee dans \cite{PPH}. 

\begin{lem}\label{lem: reg}
Soient $f_1(z),\ldots,f_n(z)\in \Q\{z\}$  des fonctions solutions d'un syst\`eme du type (\ref{eq: systeme}).  Alors l'extension de corps 
$L:= \Q(z)(f_1(z),\ldots,f_n(z))$ est r\'eguli\`ere. 
\end{lem}

\begin{proof}
Comme $L$ est finiment engendr\'e, toute sous-extension $\Q(z)\subset L'\subset L$ l'est \'egalement. Il s'agit d'un r\'esultat classique 
(voir par exemple Lang \cite[Exercise 4, Chap. VIII]{La}). Consid\'erons $L'$ la cl\^oture alg\'ebrique de $\Q(z)$ dans $L$. 
Comme $L'$ est alg\'ebrique et finiment engendr\'ee, on en d\'eduit que $L'$ est de degr\'e fini sur $\Q(z)$, disons $d$. 
Soit $f \in L'$. Comme $f$ est dans $L$, la d\'efinition du syst\`eme (\ref{eq: systeme}) 
implique que l'on a \'egalement $f(z^{q^l}) \in L$  pour tout $l\geq 0$. D'autre part,  les fonctions $f(z^{q^l})$ sont alg\'ebriques 
puisque $f(z)$ l'est.  
Ainsi, les fonctions $f(z^{q^l})$, $l\geq 0$, sont toutes dans $L'$ qui est de degr\'e $d$ sur $\Q(z)$ et il existe donc une relation lin\'eaire sur $\Q(z)$ entre 
les fonctions $f(z),f(z^q),\ldots,f(z^{q^d})$, ce qui revient \`a dire que $f(z)$ est $q$-mahl\'erienne. 
Comme $f(z)$ est \'egalement alg\'ebrique, un r\'esultat classique (voir par exemple Nishioka \cite[Theorem 5.1.7]{Ni_Liv}) 
implique que $f(z)\in \Q(z)$, comme souhait\'e. 
\end{proof}


\section{Version homog\`ene du th\'eor\`eme de Philippon}\label{sec: pphhom}

L'objet de cette section est de d\'emontrer le th\'eor\`eme  \ref{thm: pphHomogene}. 
Pour d\'emontrer ce r\'esultat, nous allons en fait commencer par prouver le corollaire \ref{coro: pphLineaire} de l'introduction 
dont l'\'enonc\'e est le suivant. 

\begin{theo}
\label{thm: pphLineaire2}
Soit $\alpha$, $0<\vert\alpha\vert<1$, un point alg\'ebrique r\'egulier pour le syst\`eme (\ref{eq: systeme}). On a : 
$$
{\rm Rel}_{\Q}(f_1(\alpha),\ldots,f_n(\alpha)) = \ev_\alpha \left({\rm Rel}_{\Q(z)}(f_1(z),\ldots,f_n(z)) \right) \, .
$$
\end{theo}

Nous utiliserons pour la d\'emonstration de ce r\'esultat le lemme d'al\`egbre lin\'eaire suivant, dont une d\'emonstration est donn\'ee dans \cite[Lemma 3.1]{Beu06}.

\begin{lem}\label{alglineaire}
Soient $f_1(z),\ldots,f_n(z)$ appartenant \`a $\Q[[z]]$ et notons $d$ la dimension de l'espace vectoriel  ${\rm Rel}_{\Q(z)}(f_1(z),\ldots,f_n(z))$. 
Alors, il existe des polyn\^omes $\lambda_{i,j}(z)\in \Q[z]$, $1\leq i\leq d, 1\leq j\leq n$, tels que les vecteurs $(\lambda_{i,1}(z),\ldots,\lambda_{i,n}(z))$, $1\leq i\leq d$ 
forment une base des $\Q[z]$-relations entre $f_1(z),\ldots,f_n(z)$, et tels que pour tout $\xi\in\Q$ le rang de la matrice 
$$
\left(\begin{array}{ccc} \lambda_{1,1}(\xi) & \cdots & \lambda_{1,n}(\xi)
\\ \vdots && \vdots
\\  \lambda_{d,1}(\xi) &    \cdots & \lambda_{d,n}(\xi)\end{array}\right)
$$
est \'egal \`a $d$. 
\end{lem}

\begin{proof}[D\'emonstration du th\'eor\`eme \ref{thm: pphLineaire2}]
Nous prouverons tout d'abord le r\'esultat dans le cas o\`u les fonctions sont lin\'eairement ind\'ependantes. 
Nous montrerons ensuite comment l'on peut, dans le cas g\'en\'eral, se ramener \`a cette premi\`ere situation.

\medskip

Supposons dans un premier temps que $\dim {\rm Rel}_{\Q(z)}(f_1(z),\ldots,f_n(z)) = 0$, c'est-\`a-dire que les fonctions $f_1(z),\ldots,f_n(z)$ 
sont lin\'eairement ind\'ependantes sur $\Q(z)$.  On va maintenant raisonner par l'absurde en supposant  qu'il existe un $n$-uplet de nombres alg\'ebriques 
$\lambda_1,\ldots,\lambda_n$, non tous nuls, 
tel que
\begin{equation}
\label{dependancevaleurs}
\sum \lambda_i f_i(\alpha) = 0\, ,
\end{equation}
o\`u $\alpha$, $0<\vert\alpha\vert<1$,  d\'esigne un point alg\'ebrique r\'egulier pour le syst\`eme mahl\'erien associ\'e aux fonctions $f_i$. 
D'apr\`es le th\'eor\`eme \ref{thm: pph}, il existe des polyn\^omes $p_1(z),\ldots,p_n(z), r(z) \in \Q[z]$, premiers entre eux, tels que 
\begin{equation}
\label{dependanceaffine}
\sum_{i=1}^n p_i(z) f_i(z) = r(z),
\end{equation}
$p_i(\alpha)=\lambda_i$, $1\leq i\leq n$, et $r(\alpha) = 0$. Sans perte de g\'en\'eralit\'e, on peut supposer que $r(\alpha^q) \neq 0$. 
En effet, si ce n'est pas le cas, il est toujours possible de consid\'erer 
le plus petit entier $l$ tel que $r(\alpha^{q^{l+1}}) \neq 0$. L'\'equation de d\'ependance affine \eqref{dependanceaffine} induit alors une relation de d\'ependance 
lin\'eaire non triviale
$$ 
\sum_{i=1}^n p_i(\alpha^{q^l})f_i(\alpha^{q^l}) = 0 
$$
et on applique le raisonnement qui suit \`a $\alpha^{q^l}$.

 On va construire \`a pr\'esent un nouveau syst\`eme en rempla\c cant l'une des fonctions $f_i$ par $r$.
Il existe un indice $i$ pour lequel $p_i(\alpha^q)$ est non nul. En effet, si $p_i(\alpha^q) = 0$ pour chaque $i$, comme les fonctions $f_i$ sont toutes 
d\'efinies en $\alpha^q$ car $\alpha$ est un point r\'egulier, on obtiendrait que $r(\alpha^q) = 0$, ce qui serait contraire \`a notre hypoth\`ese. 
Quitte \`a permuter les indices, on peut supposer que $p_n(\alpha^q)\neq 0$. On consid\`ere alors la matrice suivante : 
$$ 
S(z) := \left(\begin{array}{cccc} 1 & &  &  \\ & \ddots & & \\ & & 1 & \\ p_1(z) & p_2(z) & \cdots & p_n(z) \end{array} \right)
$$ 
de sorte que
\begin{equation}
S(z).\left(\begin{array}{c} f_1(z) \\ \vdots \\ f_n(z) \end{array} \right) = \left(\begin{array}{c} f_1(z) \\ \vdots \\ f_{n-1}(z) \\ r(z) \end{array} \right).
\end{equation}
Notons que la matrice $S(z)$ n'a pas de p\^ole et a pour d\'eterminant le polyn\^ome $p_n(z)$. 
Par construction, le vecteur $(f_1(z),\ldots,f_{n-1}(z),r(z))^T$ est solution du syst\`eme associ\'e \`a la matrice 
$$
B(z) := S(z) A(z) S(z^q)^{-1} \,.
$$
Notons que  $B$ est bien d\'efinie en $\alpha$. On a par ailleurs, l'\'egalit\'e suivante pour les d\'eterminants : 
\begin{equation}
\label{egalitedet}
{\rm det}(B(z)) := {\rm det}(S(z)){\rm det}(A(z)){\rm det}(S(z^q))^{-1} = {\rm det}(A(z))\frac{p_n(z)}{p_n(z^q)} \, \cdot
\end{equation}
Comme par hypoth\`ese,  les fonctions $f_1(z),\ldots,f_n(z)$ sont lin\'eairement ind\'ependantes sur $\Q(z)$ et que la matrice $B(z)$ est inversible, 
on obtient que les fonctions $f_1(z^q),\ldots,f_{n-1}(z^q)$ et $r(z^q)$ sont \'egalement lin\'eairement ind\'ependantes sur $\Q(z)$. 
Ainsi, la seule relation lin\'eaire sur $\Q(z)$ liant $r(z)$ et  $f_1(z^q),\ldots,f_{n-1}(z^q), r(z^q)$ est, \`a multiplication par une constante pr\`es, 
la relation banale : 
$$
r(z) = \frac{r(z)}{r(z^q)} r(z^q) \, \cdot
$$
On en d\'eduit que la  $n$-i\`eme ligne de la matrice $B(z)$ est,  \`a multiplication par une constante pr\`es, \'egale \`a : 
$$
\left(0,\ldots,0, \frac{r(z)}{r(z^q)} \right) \, .
$$
Puisque $r(\alpha) = 0$ et, par hypoth\`ese, $r(\alpha^q) \neq 0$, on obtient que 
\begin{equation*}
\label{noyau}
(0,\ldots,0,1)B(\alpha) = 0 \, .
\end{equation*}
Cependant, on a 
\begin{align*}
(0,\ldots,0,1)B(\alpha)& = (0,\ldots,0,1)S(\alpha)A(\alpha)S(\alpha^q)^{-1} \\
         & = (p_1(\alpha),\ldots,p_n(\alpha))A(\alpha)S(\alpha^q)^{-1} \\    
         &  \neq 0
\end{align*}          
car la matrice $A(\alpha)S(\alpha^q)^{-1}$ est inversible et que les $p_i(\alpha)$ ne sont pas tous nuls. 
On obtient donc une contradiction, ce qui prouve le th\'eor\`eme dans le cas o\`u $\dim {\rm Rel}_{\Q(z)}(f_1(z),\ldots,f_n(z)) = 0$. 

\medskip

Supposons \`a pr\'esent que $\dim {\rm Rel}_{\Q(z)}(f_1(z),\ldots,f_n(z)) = d \geq 1$.  
D'apr\`es le lemme \ref{alglineaire}, on peut choisir une famille de vecteurs 
$$
\lambd_i(z):=(\lambda_{i,1}(z),\ldots,\lambda_{i,n}(z)), \;\; 1\leq i\leq d,  
$$ 
dont les coordonn\'ees sont dans $\Q[z]$ et tels que le rang de la matrice 
$$
M(\xi):=\left(\begin{array}{ccc} \lambda_{1,1}(\xi) & \cdots & \lambda_{1,n}(\xi)
\\ \vdots && \vdots
\\  \lambda_{d,1}(\xi) &    \cdots & \lambda_{d,n}(\xi)\end{array}\right)
$$ 
est \'egal \`a $d$ pour tout $\xi$ dans $\Q$.
Quitte \`a renum\'eroter les fonctions $f_i$, on peut donc supposer que le mineur principal de la matrice 
$$
\left(\begin{array}{ccc} \lambda_{1,1}(\alpha) & \cdots & \lambda_{1,n}(\alpha)
\\ \vdots && \vdots
\\  \lambda_{d,1}(\alpha) &    \cdots & \lambda_{d,n}(\alpha)\end{array}\right)
$$
est inversible. On consid\`ere alors la matrice suivante 
$$
S(z) := \left(\begin{array}{cccccc} \lambda_{1,1}(z) & \cdots & \cdots & \cdots & \cdots & \lambda_{1,n}(z)
\\ \vdots &&&&& \vdots
\\  \lambda_{d,1}(z) & \cdots & \cdots & \cdots &    \cdots & \lambda_{d,n}(z)
\\ 0 &\cdots& 0& 1 & 0  & \cdots
\\ \vdots & & & \ddots & \ddots  & \vdots
\\ 0 & \cdots & \cdots & \cdots & 0 & 1 \end{array}\right)
$$
Cette matrice est donc d\'efinie et inversible en $z=\alpha$. D'autre part, on a 
\begin{equation}\label{eq: S(z)}
S(z) \left(\begin{array}{c} f_1(z) \\ \vdots \\ f_n(z) \end{array} \right) = \left(\begin{array}{c} 0 \\ \vdots \\ 0 \\ f_{d+1}(z) \\ \vdots \\ f_n(z) \end{array} \right)\, .
\end{equation}
Fixons un entier $l_0$ tel que le d\'eterminant de $S(z)$ ne s'annule en aucun des points $\alpha^{q^l}$, pour $l\geq l_0$, et notons $q_0=q^{l_0}$. 
Consid\'erons enfin la matrice 
$$
B(z) = S(z)A(z)S(z^{q_0})^{-1} \,.
$$
On a l'\'equation malh\'erienne suivante
$$
\left(\begin{array}{c} 0 \\ \vdots \\ 0 \\ f_{d+1}(z) \\ \vdots \\ f_n(z) \end{array} \right) = B(z) \left(\begin{array}{c} 0 
\\ \vdots \\ 0 \\ f_{d+1}(z^{q_0}) \\ \vdots \\ f_n(z^{q_0}) \end{array} \right)
$$ 
pour laquelle le point $\alpha$ est un point r\'egulier. D'autre part, l'ind\'ependance lin\'eaire des fonctions $f_{d+1}(z),\ldots,f_n(z)$ nous garantit que la matrice 
$B(z)$ est triangulaire inf\'erieure, de la forme suivante
$$B(z) = \left(\begin{array}{ccc|ccccc} & & & & & & & \\ & D(z) & & & & 0 & & \\ & & & & & & & 
\\ \hline & & & & & & & \\ & & & & & & & \\ 
& E(z) & & & & C(z) & & \\& & & & & & & \\ & & & & & & & \end{array}\right)
$$ o\`u $C(z)$ est une matrice carr\'ee de taille $n-d$. 
On consid\`ere alors le sous-syst\`eme 
$$
\left(\begin{array}{c} f_{d+1}(z) \\ \vdots \\ f_n(z) \end{array} \right) = C(z) \left(\begin{array}{c} f_{d+1}(z^{q_0}) \\ \vdots \\ f_n(z^{q_0}) \end{array} \right)\,.
$$ 
Le point $\alpha$ est encore r\'egulier pour ce syst\`eme. En effet, par construction, pour tout entiel $l\geq 1$, $\alpha^{q_0^l}$ n'est p\^ole d'aucun des coefficients de $B(z)$ et 
donc \textit{a fortiori} d'aucun des coefficients de $C(z)$. 
D'autre part 
$$
\det B(z) = \det C(z)\det D(z)\,,
$$
et $\alpha^{q_0^l}$ n'\'etant ni un z\'ero de $\det B(z)$, ni un p\^ole de $\det D(z)$, $\alpha^{q_0^l}$ 
n'est pas un z\'ero de $\det C(z)$.  
Consid\'erons maintenant un vecteur 
$$ 
\lambd:= (\lambda_1,\ldots,\lambda_n) \in {\rm Rel}_{\Q}(f_1(\alpha),\ldots,f_n(\alpha))\, .
$$ 
La matrice $S(z)$ \'etant inversible en $z=\alpha$, on peut consid\'erer le vecteur 
$\boldsymbol{\mu} :=  \lambd S(\alpha)^{-1}$. D'apr\`es \eqref{eq: S(z)}, on obtient que $\boldsymbol\mu$   
appartient \`a l'ensemble 
$$
{\rm Rel}_{\Q}(0,\ldots,0,f_{d+1}(\alpha),\ldots,f_n(\alpha))\, .
$$
Notons $\boldsymbol\mu := (\mu_1,\ldots,\mu_n)$,  
de sorte que 
$$
(\mu_{d+1},\ldots,\mu_n) \in {\rm Rel}_{\Q}(f_{d+1}(\alpha),\ldots,f_n(\alpha))\, .
$$
Par hypoth\`ese, le syst\`eme 
$$
\left(\begin{array}{c} f_{d+1}(z) \\ \vdots \\ f_n(z) \end{array} \right) = C(z) \left(\begin{array}{c} f_{d+1}(z^{q_0}) \\ \vdots \\ f_n(z^{q_0}) \end{array} \right)
$$ 
est form\'e de fonctions lin\'eairement ind\'ependantes et admet $\alpha$ comme point r\'egulier.
La premi\'ere partie de la preuve montre donc que 
$$
\mu_{d+1} = \cdots = \mu_n = 0\,.
$$
En posant 
$$
\tilde{S}(z) := 
S(z) - \left(\begin{array}{ccc|ccc} &&&&&\\ & 0_{d\times d}& & & 0_{d\times (n-d)} & \\ &&&&&\\ \hline &&&&&\\ &0_{(n-d)\times d} 
&&& {\rm I}_{n-d}& \\ &&&&&\end{array}\right)\, ,
$$
on obtient alors
\begin{equation}\label{eq: lsa}
\lambd =  \boldsymbol\mu S(\alpha) = \boldsymbol\mu\tilde{S}(\alpha)\, .
\end{equation}
Par construction, chaque ligne de la matrice $\tilde{S}(z)$ appartient \`a l'espace vectoriel 
$ {\rm Rel}_{\Q(z)}(f_1(z),\ldots,f_n(z))$   
et donc le vecteur $ \boldsymbol\mu \tilde{S}(z)$ appartient \`a ${\rm Rel}_{\Q(z)}(f_1(z),\ldots,f_n(z))$.
On d\'eduit donc de \eqref{eq: lsa} que 
$$
\lambd  \in  \ev_\alpha \left({\rm Rel}_{\Q(z)}(f_1(z),\ldots,f_n(z)) \right) \, ,
$$
ce qui ach\`eve cette d\'emonstration. 
\end{proof}

Nous sommes \`a pr\'esent en mesure de prouver le th\'eor\`eme \ref{thm: pphHomogene}.

\begin{proof}[Preuve du th\'eor\`eme \ref{thm: pphHomogene}]
Soient $\alpha$, $0<\vert\alpha\vert<1$,  un point alg\'ebrique r\'egulier pour le syst\`eme \eqref{eq: systeme} 
et $P \in\Q[X_1,\ldots,X_n]$ un polyn\^ome homog\`ene de degr\'e $d\geq 1$ en 
les variables $X_1,\ldots,X_n$ tel que 
$$
P(f_1(\alpha),\ldots,f_n(\alpha))=0 \,.
$$
Notons $M_1(X_1,\ldots,X_n),\ldots,M_N(X_1,\ldots,X_n)$ une \'enum\'eration des mon\^{o}mes de degr\'e $d$ en les variables $X_1,\ldots,X_n$. On consid\`ere alors 
les fonctions $g_1(z):=M_1(f_1(z),\ldots,f_n(z)),\ldots, g_N(z):=M_N(f_1(z),\ldots,f_n(z))$.  
Celles-ci sont solutions du syst\`eme  mahl\'erien 
\begin{equation*}
\left(\begin{array}{c} g_1(z) \\ \vdots \\ g_N(z) \end{array} \right)  = B(z) \left(\begin{array}{c} g_1(z^k) \\ \vdots \\ g_N(z^k) \end{array} \right) \,,
\end{equation*}
o\`u $B(z)$ est une sous-matrice de $A(z)^{\oplus d}$, puissance $d$-i\`eme de Hadamard de la matrice $A(z)$. 
En particulier les racines du d\'eterminant de $B(z)$ sont celles de celui de $A(z)$ et les p\^oles des coefficients de $B(z)$ sont ceux des coefficients de $A(z)$. 
Le point $\alpha$ est donc un point r\'egulier pour ce syst\`eme. En appliquant le th\'eor\`eme \ref{thm: pphLineaire2} \`a ce  syst\`eme, on obtient l'existence 
de polyn\^{o}mes $v_1(z),\ldots,v_N(z)$ tels que
$$
\sum_{i=1}^N v_i(z)g_i(z) = 0 \qquad \text{ et }  P(X_1,\ldots,X_n)   = \sum_{i=1}^N v_i(\alpha)M_i(X_1,\ldots,X_n)\, .
$$
On pose alors $Q(z,X_1,\ldots,X_n):=\sum_{i=1}^n v_i(z)M_i(X_1,\ldots,X_n)$, ce qui termine la d\'emonstration.
\end{proof}


\section{Premi\`ere d\'emonstration du point (i) du th\'eor\`eme \ref{thm: baker}}\label{sec: pointi}

Nous allons montrer comment obtenir le point (i) du th\'eor\`eme \ref{thm: baker} \`a partir du th\'eor\`eme \ref{thm: pph} de Philippon.  
Pour cela, nous aurons besoin de  trois r\'esultats auxiliaires.

Le lemme suivant est une adaptation directe d'une construction utilis\'ee par Bell, Bugeaud et Coons dans \cite{BBC}.  
Elle permet, \'etant donn\'e un nombre alg\'ebrique $\alpha$ non nul, de plonger un syst\`eme mahl\'erien dont les solutions sont d\'efinies au point $\alpha$ dans 
un m\'eta-syst\`eme pour lequel les points $\alpha^{q^l}$, $l\geq 0$,  ne sont jamais des p\^oles des coefficients de la matrice associ\'ee.  
Notons que cette astuce s'av\`ere inutile dans le cas des s\'eries automatiques (ou m\^eme r\'eguli\`eres) puisque l'on peut alors se ramener 
\`a un syst\`eme mahl\'erien dont la matrice est \`a coefficients polyn\^omes.

\medskip

\begin{lem}\label{lem: bbc} 
Consid\'erons un syst\`eme du type (\ref{eq: systeme})  et  $\alpha\in\Q$, $0<\vert \alpha\vert<1$ qui ne soit p\^ole d'aucune des fonctions 
$f_1(z),\ldots,f_n(z)$.   Soit ${\bf k}$ un corps de nombres contenant les coefficients des $f_i(z)$ et $\alpha$. 
Alors, il existe un syst\`eme mahl\'erien d'ordre $m\geq n$, ayant pour solution des fonctions $g_1(z),\ldots,g_m(z)$ et tel que :
\begin{itemize}

\smallskip

\item[{\rm(i)}]  la matrice $B(z)$ associ\'ee \`a ce nouveau syst\`eme n'a de p\^ole en aucun des points $\alpha^{q^l}$, pour tout entier $l$, 

\smallskip

\item[{\rm(ii)}]  $g_i(\alpha)=\lambda_if_i(\alpha)$, o\`u $\lambda_i\in {\bf k}\setminus\{0\}$, pour tout $i\in\{1,\ldots,n\}$\, .

\end{itemize} 
\end{lem}

\medskip

\begin{proof}
Notons $A_0(z)$ la matrice associ\'ee \`a notre syst\`eme.   
Sans perte de g\'en\'eralit\'e, nous pouvons supposer qu'il existe un entier $l$ tel que la matrice $A_0(z)$ ne soit pas d\'efinie au point $\alpha^{q^l}$. 
En effet, si ce n'est pas le cas, le syst\`eme de d\'epart jouit d\'ej\`a des propri\'et\'es requises. 
Notons $b(z)$ le polyn\^ome obtenu comme ppcm des d\'enominateurs des coefficients de la matrice $A_0(z)$.   On a alors $A_0(z) = A(z)/b(z)$ pour une matrice $A(z)$ 
dont les coefficients sont des \'el\'ements de ${\bf k}[z]$. 
Notons qu'il existe un entier $n_0$ tel que 
\begin{equation}\label{eq: bki}
b(\alpha^{q^l}) \not = 0 \,
\end{equation}
pour tout $l\geq n_0$. 
D'autre part, on peut \'ecrire $b(z)=\gamma z^r\beta(z)$ o\`u $r$ est un entier, $\gamma$ est un \'el\'ement de 
${\bf k}$ et $\beta(z)$ est un polyn\^ome \`a coefficients dans $\bf k$ tel que $\beta(0)=1$. 
Pour tout $i=1,\ldots,n$, on pose :
$$
h_i(z) := f_i(z) \prod_{j\geq 0} \beta(z^{q^j})\, .
$$
Les fonctions $h_1(z),\ldots,h_n(z)$ sont ainsi solutions du syst\`eme mahl\'erien associ\'e \`a la matrice $A(z)/\gamma z^r$. 
En d\'erivant, on obtient un nouveau syst\`eme de la forme
\begin{equation}\label{eq: der}
\left( \begin{array}{ c }
     h_1(z) \\
     \vdots \\
     h_n(z)\\
     h'_1(z)\\
      \vdots \\
     h'_n(z)\\
  \end{array} \right) = \frac{1}{\gamma z^r}  \left(\begin{array}{cc} 
A(z) & 0 \\

\star & qz^{q-1}A(z)
\end{array}\right)   \left( \begin{array}{ c }
      h_1(z^q) \\
     \vdots \\
     h_n(z^q)\\
     h'_1(z^{q})\\
      \vdots \\
     h'_n(z^{q})\\
  \end{array} \right) \, ,
\end{equation}
o\`u $\star$ d\'esigne une matrice $n\times n$ \`a coefficients dans ${\bf k}[z]$. En it\'erant ce proc\'ed\'e $t$ fois, 
on obtient un syst\`eme de la forme :
\begin{equation}\label{eq: iter}
\left( \begin{array}{ c }
     h_1(z) \\
     \vdots \\
     h_n(z)\\
     h'_1(z)\\
      \vdots \\
     h'_n(z)\\
     \vdots \\
     h_1^{(t)}(z)\\
      \vdots \\
     h_n^{(t)}(z)\\
  \end{array} \right) = \frac{1}{\gamma z^r}   \left( \begin{array}{cccc} 
  A(z)  && &\\ 
\star & qz^{q-1} A(z) &&\\
 \vdots & \ddots& \ddots &\\
 \star & \ldots & \star & q^t z^{t(q-1)} A(z)
\end{array}\right)  \left( \begin{array}{ c }
      h_1(z^q) \\
     \vdots \\
     h_n(z^q)\\
     h'_1(z^{q})\\
      \vdots \\
     h'_n(z^{q})\\
     \vdots \\
     h_1^{(t)}(z^q)\\
      \vdots \\
     h_n^{(t)}(z^{q})\\
  \end{array} \right) \, .
\end{equation}
Notons $s$ l'ordre de la racine $\alpha$ dans le polyn\^ome $\prod_{j=0}^{n_0-1} \beta(z^{q^j})$ et 
$T(z)$ le polyn\^ome \`a coefficients dans $\bf k$ d\'efini par l'\'egalit\'e 
$$
\prod_{j=0}^{n_0-1} \beta(z^{k^j}) = (z-\alpha)^s T(z) \, .
$$
Notons que $T(\alpha)\not=0$. 
Un calcul montre alors que, pour tout $i=1,\ldots,n$, on a 
\begin{equation}\label{eq: lem5}
f_i(\alpha) s! T(\alpha) = \frac{h_i^{(s)}(\alpha)}{\prod_{j\geq n_0} \beta(\alpha^{q^j})} \, \cdot
\end{equation}
D'autre part, on peut v\'erifier que les fonctions
 $$
 h_{i,j}(z):= \frac{h_i^{(j)}(z)}{\prod_{l\geq n_0} \beta(z^{q^l})}, \; i=1,\ldots,n, j=0,\ldots, s,
 $$
 satisfont \`a un syst\`eme mahl\'erien d'ordre $m:=n(s+1)$ associ\'e \`a la matrice
 $$
B(z):=\frac{1}{\gamma z^r \beta(z^{q^{n_0}})}   \left( \begin{array}{cccc} 
  A(z)  && &\\ 
\star & qz^{q-1} A(z) &&\\
 \vdots & \ddots& \ddots &\\
 \star & \ldots & \star & q^s z^{s(q-1)} A(z)
\end{array}\right) \, \cdot
 $$
 Au vu de (\ref{eq: bki}) et (\ref{eq: lem5}), et puisque les $\star$ d\'esignent des sous-matrices \`a coefficients dans ${\bf k}[z]$, 
 ce syst\`eme a toutes les propri\'et\'es requises en posant $g_{al+i}(z):= h_{i,s-a}(z)$ pour $a=0,\ldots,s$ et $i=1,\ldots,n$. 
 \end{proof}

\medskip

Le lemme suivant est la cons\'equence d'une autre construction permettant de plonger un syst\`eme mahl\'erien pour lequel 
un point $\alpha$ est singulier (mais tel que $\alpha^{q^l}$ n'est jamais un p\^ole de la matrice associ\'ee) 
dans un m\'eta-syst\`eme pour lequel $\alpha$ est un point r\'egulier. 

\medskip

\begin{lem}\label{lem: dedoublement} 
Consid\'erons un syst\`eme du type (\ref{eq: systeme})  et  $\alpha\in\Q$, $0<\vert \alpha\vert<1$ qui ne soit p\^ole d'aucune des fonctions 
$f_1(z),\ldots,f_n(z)$.   
Supposons en outre que, pour tout entier $l$, $\alpha^{q^l}$ ne soit pas un p\^ole de la matrice $A(z)$.  
Alors, il existe un syst\`eme mahl\'erien d'ordre $m\geq n$, admettant des solutions $h_1(z),\ldots,h_m(z)$, et tel que :  
\begin{itemize}

\smallskip

\item[{\rm(i)}]  $\alpha$ est un point r\'egulier pour ce syst\`eme, 

\smallskip

\item[{\rm(ii)}]  $h_i(z)=f_i(z)$ pour $i=1\ldots,n$. 

\end{itemize} 
\end{lem}

\begin{proof}
Pla\c cons-nous dans les conditions du th\'eor\`eme et supposons que le point 
$\alpha$ est singulier (mais tel que la matrice $A(z)$ associ\'ee soit d\'efinie en tout point de la forme 
$\alpha^{q^l}$, comme le permet l'hypoth\`ese). L'id\'ee est alors de \og d\'edoubler le syst\`eme\fg{}
en remarquant que les $2n$ fonctions $f_1(z),\ldots,f_n(z),f_1(z^q),\ldots,f_n(z^q)$ 
satisfont \`a la relation   
\begin{equation}\label{eq: double}
\left( \begin{array}{ c }
     f_1(z) \\
     \vdots \\
     f_n(z)\\
     f_1(z^q)\\
      \vdots \\
     f_n(z^q)\\
  \end{array} \right) = B_1(z)\left( \begin{array}{ c }
     f_1(z^q) \\
     \vdots \\
     f_n(z^q)\\
     f_1(z^{q^2}) \\
     \vdots \\
     f_n(z^{q^2})
  \end{array} \right) \, ,
\end{equation}
o\`u $B_1(z)$ est donn\'e par la matrice suivante :
$$\left(\begin{array}{c|c} \begin{array}{cccc} 
\\ &  A(z) - I_n& &
\\ \\
 \end{array} & A(z^q) 
\\ \hline
\\I_n & \begin{array}{ccc} 0 
 \end{array}
\end{array}\right) \, .$$ 
Notons d'abord que par hypoth\`ese la matrice $B_1(z)$ est bien d\'efinie en tout point de la forme $\alpha^{q^l}$. Ainsi, si $\alpha$ est singulier,  
cela signifie n\'ecessairement qu'il existe un entier $l$ tel que $\det B_1(\alpha^{q^l})=0$.  
D'autre part, on a $\det B_1(z)=\det A(z^q)$.  Comme les racines de $\det B_1(z)$ sont en nombre fini,  
il existe un entier $n_0$ tel que  $\det B_1(\alpha^{q^l})\not=0$ pour tout $l \geq n_0$.  
En it\'erant $n_0$ fois la m\'ethode de d\'edoublement, on obtient un syst\`eme du type:
\begin{equation}\label{eq: iter}
\left( \begin{array}{ c }
     f_1(z) \\
     \vdots \\
     f_n(z)\\
     f_1(z^q)\\
      \vdots \\
     f_n(z^q)\\
     \vdots \\
     f_1(z^{q^{n_0}})\\
      \vdots \\
     f_n(z^{q^{n_0}})\\
  \end{array} \right) = B_{n_0}(z)\left( \begin{array}{ c }
      f_1(z^q) \\
     \vdots \\
     f_n(z^q)\\
     f_1(z^{q^2})\\
      \vdots \\
     f_n(z^{q^2})\\
     \vdots \\
     f_1(z^{q^{n_0+1}})\\
      \vdots \\
     f_n(z^{q^{n_0+1}})\\
  \end{array} \right) \, ,
\end{equation}
o\`u  $B_{n_0}(z)\in {\rm GL}_{n2^{n_0}}(\Q(z))$ est d\'efinie r\'ecursivement par : 
$$
B_j(z) := \left(\begin{array}{c|c} \begin{array}{cccc} 
\\ &  B_{j-1}(z) - I_{n2^{j-1}}& &
\\ \\
 \end{array} & B_{j-1}(z^q) 
\\ \hline
\\I_{n2^{j-1}} & \begin{array}{ccc} 0 
 \end{array}
\end{array}\right) \, 
$$ 
et $B_0(z):=A(z)$. 
On obtient donc que $\det B_{n_0}(z) = \det B_0(z^{q^{n_0}})$, ce qui implique que le point $\alpha$ est r\'egulier pour le syst\`eme associ\'e \`a la matrice 
$B_{n_0}(z)$. En posant $m=n2^{n_0}$ et $h_{an+i}(z) := f_i(z^{q^{a}})$ pour $i=1,\ldots,m$ et $a=0,\ldots, 2n_0$, 
on obtient le r\'esultat souhait\'e.  
\end{proof}

Nous aurons \'egalement besoin du lemme de descente suivant. 

\begin{lem}\label{lem: descente}
Soient $h_1(z),\ldots,h_n(z)$ des fonctions analytiques dans un voisinage de l'origine et \`a coefficients dans un corps de nombres ${\bf k}$. Soit 
$\alpha\in {\bf k}$  dans le domaine de 
convergence de ces fonctions.  
Supposons qu'il existe $w_1(z),\ldots,w_n(z)\in \Q[z]$, non tous nuls, tels que 
\begin{equation}\label{eq: lem1}
w_1(z)h_1(z) + \ldots + w_n(z)h_n(z) =0 \, ,
\end{equation}
avec $w_i(\alpha)=0$ pour tout $i$ dans un ensemble $ \mathcal I$ et  $w_{i_0}(\alpha)\not= 0$ pour un certain indice $i_0\in\{1,\ldots,n\}\setminus \mathcal I$.   
Alors il existe $w'_1(z),\ldots,w'_n(z)\in {\bf k}[z]$, non tous nuls, tels que 
$$
w'_1(z)h_1(z) + \ldots + w'_n(z)h_n(z) =0 \, ,
$$
avec $w'_i(\alpha)=0$ pour tout $i$ dans  $\mathcal I$ et $w'_{i_0}(\alpha)\not=0$. 
\end{lem}

\medskip

\begin{proof} 
Les polyn\^omes $w_1(z),\ldots,w_n(z)$  \'etant en nombre fini et \`a coefficients dans $\Q$, leurs coefficients engendrent une extension de corps de degr\'e fini, disons $h$, sur 
${\bf k}$. Notons ${\bf k}_0\subset \mathbb C$ une telle extension et $m$ le maximum des degr\'es des $w_i(z)$. 
Soit $\varphi$ tel que  ${\bf k}_0={\bf k}(\varphi)$.  
Chaque polyn\^ome $w_i(z)$ peut donc s'\'ecrire sous la forme :
$$
w_i(z) = \sum_{k=0}^m \left(\sum_{l=0}^{h-1} \lambda(i,k,l) \varphi^l \right)z^k\, ,
$$
o\`u les coefficients $\lambda(i,k,l)$ sont tous dans ${\bf k}$.  

Soit $i_0$ comme dans l'\'enonc\'e. 
Comme $w_{i_0}(\alpha)\not= 0$ et que $1,\varphi,\ldots,\varphi^{h-1}$ sont lin\'eairement ind\'ependants sur ${\bf k}$,  
il existe un indice $l_0$ tel que 
\begin{equation}\label{eq: palpha}
\sum_{k=0}^m \lambda(i_0,k,l_0) \alpha^k
\not= 0.
\end{equation}

La relation (\ref{eq: lem1}) donne une relation de d\'ependance lin\'eaire sur ${\bf k}_0$ entre les fonctions $z^kh_i(z)$, $k=0,1,\ldots m$, $i=1,\ldots,n$, \`a savoir : 
 $$
\sum_{l=0}^{h-1}\left(\sum_{k=0}^m \sum_{i=1}^n \lambda(i,k,l) z^kh_i(z) \right)\varphi^l    = 0\, .
$$ 
Comme l'ind\'etermin\'ee $z$ est transcendante sur $\mathbb C$ et que les coefficients des $h_i(z)$ sont des \'el\'ements de ${\bf k}$ , on  
 en d\'eduit que pour tout $l$ :
$$
\sum_{k=0}^m\sum_{i=1}^n \lambda(i,k,l)z^kh_i(z)  = 0\, .
$$ 
On obtient notamment la relation de d\'ependance lin\'eaire suivante sur ${\bf k}[z]$ entre les $h_i(z)$ :
$$
 w'_1(z)h_1(z) +\cdots + w'_n(z)h_n(z) =0\, , 
 $$
 o\`u $w'_i(z) :=  \sum_{k=0}^m \lambda(i,k,l_0)z^k\in {\bf k}[z]$.  Notons que d'apr\`es (\ref{eq: palpha}), 
 $w'_{i_0}(\alpha)\not=0$. Cela implique au passage que 
 $w'_{i_0}(z)\not=0$ et donc que la relation est non triviale. 

 Pour conclure, il ne reste donc plus qu'\`a v\'erifier que $w'_i(\alpha)=0$ pour $i\in \mathcal I$.  
Pour un tel $i$, nous savons que $w_i(\alpha)=0$, ce qui signifie que 
 $$
 \sum_{k=0}^m \left(\sum_{l=0}^{h-1} \lambda(i,k,l) \varphi^l \right)\alpha^k = 0 \, 
 $$
 c'est-\`a-dire 
 $$
  \sum_{l=0}^{h-1} \left(\sum_{k=0}^m \lambda(i,k,l) \alpha^k \right)\varphi^l = 0 \, .
 $$
 Comme $\alpha\in{\bf k}$, on obtient que pour tout $l$ :
$$
 \sum_{k=0}^m \lambda(i,k,l) \alpha^k  = 0 \, .
 $$ 
Le choix  $l=l_0$ donne l'\'egalit\'e $w'_i(\alpha)=0$, ce qui conclut la d\'emonstration. 
\end{proof}

\medskip

Nous pouvons \`a pr\'esent d\'emontrer le point (i) du th\'eor\`eme \ref{thm: baker}.

\begin{proof}[D\'emonstration du point (i) du th\'eor\`eme \ref{thm: baker}] 
Soient $f_1(z),\ldots,f_n(z)$ des fonctions $q$-mahl\'eriennes.   
Soient $\alpha\in\Q$, $0<\vert \alpha\vert <1$,  un nombre qui n'est p\^ole d'aucune de ces fonctions et ${\bf k}$ un corps de nombres contenant 
$\alpha$ ainsi que les coefficients des $f_i$.

\medskip

Comme les fonctions $f_1(z),\ldots,f_n(z)$ sont  $q$-mahl\'eriennes, on peut trouver un syst\`eme du type (\ref{eq: systeme}) 
d'ordre $m_0\geq n$, associ\'e \`a une matrice $B(z)$, et admettant des solutions  $f_1(z),\ldots,f_{m_0}(z)$ o\`u, pour 
$i=n+1,\ldots,m_0$, les fonctions $f_i(z)$ sont de nouvelles fonctions $q$-mahl\'eriennes, toutes d\'efinies en $\alpha$.   
En utilisant le Lemme \ref{lem: bbc},  on obtient un m\'eta-syst\`eme mahl\'erien, disons d'ordre $m_1\geq m_0$, pour lequel aucun des $\alpha^{q^l}$ 
n'est p\^ole d'un des coefficients de la matrice $B_1(z)$ associ\'ee \`a ce syst\`eme. 
Notons que ce changement de syst\`eme conduit \`a remplacer les fonctions 
$f_1(z),\ldots,f_{m_0}(z)$ par des fonctions  $g_1(z),\ldots,g_{m_1}(z)$ telles que, pour tout $i=1,\ldots, m_0$, on ait 
$g_i(\alpha)=\lambda_if_i(\alpha)$ avec $\lambda_i\in {\bf k}\setminus \{0\}$, $\lambda_i\not=0$. 
Une fois cette premi\`ere \'etape r\'ealis\'ee, on utilise le Lemme \ref{lem: dedoublement}. Il nous permet d'obtenir un syst\`eme mahl\'erien encore plus gros, 
disons d'ordre $m_2 \geq m_1$, pour lequel le point $\alpha$ est r\'egulier et dont les solutions $h_1(z),\ldots,h_{m_2}(z)$ v\'erifient $h_i(z)=g_i(z)$ pour $i=1\ldots,m_1$. 
Comme les $\lambda_i$ sont tous non nuls et dans $\bf k$, 
la d\'ependance lin\'eaire sur $\bf k$ 
(ou sur $\Q$) des $f_i(\alpha)$ est \'equivalente \`a celle des $g_i(\alpha)$.  Sans perte de g\'en\'eralit\'e, on peut donc supposer que $f_i(z)=g_i(z)$ pour 
$i=1,\ldots,m_0$.

\medskip

En r\'esum\'e,  nous pouvons donc supposer sans perte de g\'en\'eralit\'e qu'il existe un syst\`eme mahl\'erien, disons d'ordre $m \geq n$, dont les solutions 
$h_1(z),\ldots,h_m(z)$ v\'erifient $h_i(z)=f_i(z)$ pour 
$i=1,\ldots,n$ et tel que $\alpha$ est un point r\'egulier pour ce syst\`eme. 
Supposons \`a pr\'esent qu'il existe $w_1,\ldots,w_n\in \Q$, non tous nuls, tels que 
$$
w_1f_1(\alpha) + \ldots + w_nf_n(\alpha) =0 \, .
$$
En appliquant le th\'eor\`eme \ref{thm: pph} \`a notre syst\`eme, on obtient l'existence de polyn\^omes $w_0(z),\ldots,w_m(z)$,  
non tous nuls et appartenant \`a $\mathbb C[z]$, 
tels que 
\begin{equation*}\label{eq:fg}
w_0(z)+ w_1(z)h_1(z) +\cdots + w_m(z)h_m(z)  =0, 
 \end{equation*}
avec $w_i(\alpha) =w_i$ pour $i=1,\ldots,n$  et $w_i(\alpha)=0$  pour $ i=n+1,\ldots,m$ et  $ i=0$.   
Puisque les $w_i$ sont non tous nuls, on peut choisir un indice $i_0$, $1\leq i_0\leq n$, tel que $w_{i_0}(\alpha)\not=0$.  
Le Lemme \ref{lem: descente} nous donne alors l'existence de polyn\^omes  $w'_0(z),\ldots,w'_m(z)$, non tous nuls et appartenant \`a ${\bf k}[z]$, tels que 
$$
 w'_0(z)+w'_1(z)h_1(z) +\cdots + w'_m(z)h_m(z) =0\, , 
 $$
o\`u  $w'_i(\alpha)=0$  pour $ i=n+1,\ldots,m$ et $ i=0$, et $w'_{i_0}(\alpha)\not=0$. 
En sp\'ecialisant au point $\alpha$, il vient: 
\begin{equation}\label{eq: relalpha}
 w'_1(\alpha)f_1(\alpha) +\cdots + w'_n(\alpha)f_n(\alpha) =0 \,. 
\end{equation}
Puisque les $w'_i(z)$ sont \`a coefficients dans $\bf k$ et que $\alpha\in \bf k$, on a $w'_i(\alpha)\in \bf k$ pour tout $i$. 
Comme $w'_{i_0}(\alpha)\not=0$, la relation (\ref{eq: relalpha}) est non triviale et les nombres $f_1(\alpha),\ldots, f_n(\alpha)$ sont donc 
lin\'eairement d\'ependants sur $\bf k$, 
ce qui termine la d\'emonstration. 
\end{proof}


\section{D\'emonstrations des th\'eor\`emes  \ref{thm: baker}, \ref{thm: StructureRelationsLineaires}  
et \ref{thm: paspole}  }\label{sec: preuves}

Dans cette partie, nous utilisons la version homog\`ene du th\'eor\`eme de Philippon, \`a savoir le th\'eor\`eme \ref{thm: pphHomogene}, 
pour en d\'eduire les th\'eor\`emes  
\ref{thm: paspole},  \ref{thm: StructureRelationsLineaires}, puis   \ref{thm: baker}.

\begin{proof}[Preuve du Th\'eor\`eme \ref{thm: paspole}]
Fixons un entier $l$ comme donn\'e dans l'\'enonc\'e. On a 
\begin{equation}\label{eq: mah}
\left( \begin{array}{ c }
     f_1(z) \\
     \vdots \\
     f_n(z)
  \end{array} \right) = A_l(z)\left( \begin{array}{ c }
     f_1(z^{q^l}) \\
     \vdots \\
     f_n(z^{q^l})
  \end{array} \right)  \, ,
\end{equation}
et l'hypoth\`ese faite sur $l$ assure que le point $\alpha^{q^l}$ est r\'egulier pour ce syst\`eme. 
Dans la suite, quitte \`a remplacer $q$ par $q^l$ et $A$ par $A_l$, on supposera que $l=1$, et donc que $\alpha^q$ est un point r\'egulier pour le syst\`eme de d\'epart.  

Notons $d$ l'ordre maximal des p\^oles des coefficients de la matrice $A(z)$ au point $\alpha$. Nous allons raisonner par r\'ecurrence sur l'entier $d$.   

Si $d=0$, alors $\alpha$ n'est pas un p\^ole de $A(z)$ et il n'y a rien \`a faire.  

On suppose \`a pr\'esent 
que la propri\'et\'e est vraie jusqu'\`a l'entier $d-1$
et que $\alpha$ est un p\^ole d'ordre $d\geq 1$ pour $A(z)$. Notons que chaque coefficient de $A(z)$ peut s'\'ecrire  sous la forme 
$$
\frac{P(z)}{Q(z)(z-\alpha)^k} \, ,
$$
o\`u $P$ et $Q$ sont des polyn\^omes, $Q(\alpha)P(\alpha)\not=0$ et $k\leq d$. \`A multiplication par une constante pr\`es, 
les polyn\^omes $P$ et $Q$ sont uniques. 
En notant $T(z)$ le plus petit commun multiple de ces polyn\^omes $Q$,   
on peut d\'ecomposer la matrice $A(z)$  de la mani\`ere suivante :  
\begin{equation}
\label{DecompositionEuclidienne}
A(z) = T(z)^{-1} \left(A_0(z) + \frac{A_1}{z-\alpha} + \cdots + \frac{A_d}{(z- \alpha)^d} \right)\, ,
\end{equation}
o\`u $A_0(z)$ est une matrice \`a coefficients dans ${\bf k}[z]$  et o\`u les $A_i$, $i=1,\ldots,d$, sont des matrices \`a coefficients dans ${\bf k}$ et $A_d\not=0$.   
En multipliant la relation \eqref{eq: mah} par $(z-\alpha)^d$ et en \'evaluant en $\alpha$, on trouve
\begin{equation}
\label{annulationpole}
 A_d \left( \begin{array}{ c }
     f_1(\alpha^{q}) \\
     \vdots \\
     f_n(\alpha^{q})
  \end{array} \right) = 0 \, ,
\end{equation}
puisque les fonctions $f_i$ sont toutes d\'efinies en $\alpha$. 
Comme $\alpha^q$ est un point r\'egulier pour le syst\`eme mahl\'erien associ\'e \`a $A(z)$, le th\'eor\`eme \ref{thm: pphHomogene}  implique l'existence d'une matrice $C(z)$, 
\`a coefficients dans $\Q[z]$, telle que 
\begin{equation}\label{eq:C}
 C(z)\left( \begin{array}{ c }
     f_1(z^{q}) \\
     \vdots \\
     f_n(z^{q}) 
  \end{array} \right) = 0  \;\;\;\; 
et  \;\;\;\;  
 C(\alpha) = A_d \, .  
\end{equation}
L'argument de descente d\'ej\`a utilis\'e dans la d\'emonstration du lemme \ref{lem: descente} montre alors que l'on peut en fait choisir $C(z)$ \`a coefficients dans ${\bf k}[z]$. 
On peut donc \'ecrire 
$$
A_d = C(z) + (z-\alpha)D(z) \, ,
$$
 o\`u $D(z)$ est une matrice \`a coefficients dans ${\bf k}[z]$. 
Posons alors 
$$
B_0(z) := A(z) - T(z)^{-1}\frac{C(z)}{(z-\alpha)^d} \cdot
$$
Pour $B_0(z)$, $\alpha$ est un p\^ole d'ordre $d-1$, et on a 
\begin{equation*}
\left( \begin{array}{ c }
     f_1(z) \\
     \vdots \\
     f_n(z)
  \end{array} \right) = B_0(z)\left( \begin{array}{ c }
     f_1(z^{q}) \\
     \vdots \\
     f_n(z^{q})
  \end{array} \right)  \, .
\end{equation*}
Pour l'instant, rien ne garantit que la matrice $B_0(z)$ est r\'eguli\`ere en $\alpha^q$, ni m\^eme qu'elle est inversible. 
On montre \`a pr\'esent que l'on peut modifier l\'eg\`erement la matrice $B_0(z)$ pour que ce soit le cas. 
Par hypoth\`ese la matrice $A(z)$ est inversible sur le disque \'epoint\'e $D(0,\rho)^\star$. 
Le d\'eterminant de cette matrice est ainsi une fraction rationnelle ne s'annulant pas sur le disque \'epoint\'e $D(0,\rho)^\star$. 
Il existe donc un entier $N_0$ tel que toute matrice $E$ 
de $\mathcal M_n(\overline{\mathbb Q})$ pour laquelle il existe $\xi \in D(0,\rho)^\star$ tel que 
$$
\Vert A(\xi) - E \Vert < |\xi|^{N_0} \,,
$$ 
est inversible, o\`u $\Vert \cdot \Vert$ d\'esigne la norme infini.  
D'autre part, la matrice $C(z)$ est \`a 
coefficients dans ${\bf k}[z]$. Il existe donc un entier $M$ tel que 
$$
\Vert C(\xi)\Vert < M
$$ 
pour tout $\xi \in \overline{D(0,\rho)}$. On s'int\'eresse maintenant au d\'eveloppement en s\'erie de Laurent de la fraction rationnelle 
$(T(z)(z-\alpha)^d)^{-1}$
$$ 
\frac{1}{T(z)(z-\alpha)^d} = \sum_{l = -r}^\infty c_lz^l \,.
$$
La s\'erie $z^{-N_0} \sum_{l\geq N_0} c_l z^l$ converge absolument sur le compact $\overline{D(0,\rho)}$. 
Il existe donc un entier $N_1\geq N_0$ tel que, pour tout $\xi \in \overline{D(0,\rho)}$, 
$$
\left\vert z^{-N_0} \sum_{l\geq N_1} c_l z^l\right\vert  < \frac{1}{M} \, .
$$
Notons alors 
$$
Q(z) := \sum_{l=-r}^{N_1-1} c_l z^l \in {\bf k}[z,z^{-1}]
$$
et posons 
$$
B_1(z) := B_0(z) + Q(z)C(z) = A(z) + \left(Q(z) - \frac{1}{T(z)(z-\alpha)^d}\right) C(z) \,.
$$
Par construction on a, pour tout $\xi \in D(0,\rho)^\star$,
\begin{eqnarray*}
\left\Vert B_1(\xi) - A(\xi)\right\Vert & \leq & \left\Vert \left(Q(\xi) - \frac{1}{T(\xi)(\xi-\alpha)^d)}\right) C(\xi) \right\Vert
\\ & \leq & \left|\sum_{l \geq N_1} c_l \xi^l \right|\times \Vert C(\xi) \Vert 
\\ & < & |\xi|^{N_0}
\end{eqnarray*}
La matrice $B_1$ est donc inversible sur le disque \'epoint\'e $D(0,\rho)^\star$. Par construction de la matrice $C(z)$, la matrice $B_1(z)$ satisfait elle aussi le syst\`eme
$$\left( \begin{array}{ c }
     f_1(z) \\
     \vdots \\
     f_n(z)
  \end{array} \right) = B_1(z)\left( \begin{array}{ c }
     f_1(z^{q}) \\
     \vdots \\
     f_n(z^{q})
  \end{array} \right)  \, ,$$
  et le point $\alpha^q$ est r\'egulier pour ce syst\`eme. Enfin, $\alpha$ est un p\^ole d'ordre au maximum $d-1$, pour la matrice $B_1(z)$, comme souhait\'e. 
 
  Notons que cette d\'emonstration fournit une m\'ethode permettant de d\'eterminer une telle matrice $B_1(z)$. En effet, le th\'eor\`eme  \ref{BaseRelationsLin\'eaires1} 
  permet de calculer une matrice $C(z)$ comme en \eqref{eq:C}.  Il suffit ensuite de d\'eterminer de fa\c con explicite un entier $N_1$ convenable afin d'en 
  d\'eduire $B_1(z)$.

En appliquant l'hypoth\`ese de r\'ecurrence \`a ce syst\`eme,  on obtient l'existence d'une  matrice $B(z)$ satisfaisant aux propri\'et\'es voulues. 
Cela conclut la d\'emonstration de la premi\`ere partie du th\'eor\`eme.   

\medskip

Supposons \`a pr\'esent que les fonctions $f_1(z),\ldots,f_n(z)$ sont lin\'eairement ind\'ependantes sur ${\bf k}(z)$. D'apr\`es ce que l'on vient de d\'emontrer,  
on peut trouver une matrice $B(z)$ de sorte que 
$$
\left( \begin{array}{ c }
     f_1(z) \\
     \vdots \\
     f_n(z)
  \end{array} \right) = B(z) \left( \begin{array}{ c }
     f_1(z^{q^l}) \\
     \vdots \\
     f_n(z^{q^l})
  \end{array} \right) \, 
$$
et que $\alpha$ ne soit pas un p\^ole de $B(z)$. 
D'autre part, on a \'egalement : 
$$
\left( \begin{array}{ c }
     f_1(z) \\
     \vdots \\
     f_n(z)
  \end{array} \right) = A_l(z) \left( \begin{array}{ c }
     f_1(z^{q^l}) \\
     \vdots \\
     f_n(z^{q^l})
  \end{array} \right) \, .
$$
Par diff\'erence, on obtient donc 
$$
\left(B(z)^{-1} - A_l(z)^{-1} \right).\left( \begin{array}{ c }
     f_1(z) \\
     \vdots \\
     f_n(z)
  \end{array} \right) = 0 \, .
$$
Les fonctions $f_1(z),\ldots,f_m(z)$ \'etant lin\'eairement ind\'ependantes sur ${\bf k}(z)$, cela implique que 
$$
A_l(z) = B(z) \, ,
$$ 
comme voulu. 
\end{proof}

\begin{rem}Notons {\it a contrario} que si les fonctions $f_1(z),\ldots,f_n(z)$ sont lin\'eairement d\'ependantes sur ${\bf k}(z)$, on peut toujours les obtenir 
comme solution d'un syst\`eme mahl\'erien ayant un p\^ole au point $\alpha$. En effet, la d\'ependance lin\'eaire des fonctions $f_i$ implique 
l'existence d'une matrice non nulle $C(z)$ telle que 
$$
 C(z)\left( \begin{array}{ c }
     f_1(z) \\
     \vdots \\
     f_n(z) 
  \end{array} \right) = 0\, .
$$
Les fonctions $f_1(z),\ldots,f_n(z)$ sont alors solutions du syst\`eme mahl\'erien associ\'e \`a la matrice 
$$
B(z) := A(z) + \frac{1}{(z-\alpha)^r}C(z) \, .
$$
En choisissant l'entier $r$ assez grand, on peut toujours garantir que $B(z)\in {\rm GL}_n({\bf k}(z))$ et que $\alpha$ est un p\^ole de $B(z)$.  
\end{rem}

\medskip

Nous allons maintenant d\'emontrer le th\'eor\`eme \ref{thm: StructureRelationsLineaires}. 

\begin{proof}[Preuve du Th\'eor\`eme  \ref{thm: StructureRelationsLineaires}]
Soit $\alpha$ un nombre alg\'ebrique non nul et $l$ un entier tel que $|\alpha^{q^l} | < \rho $. 
Supposons que $\alpha$ n'est pas un p\^ole de $A_l(z)$. 
Rappelons que notre objectif est de montrer l'\'egalit\'e suivante :  
\begin{equation}
\label{StructureRelationsLineaires}
{\rm Rel}_{{\bf k}}(f_1(\alpha),\ldots,f_n(\alpha)) = {\rm ker}_{{\bf k}}A_l(\alpha) + {\rm ev}_\alpha\left({\rm Rel}_{{\bf k}(z)}(f_1(z),\ldots,f_n(z))\right) \, .
\end{equation}

Notons tout d'abord que l'inclusion $\supset$ est banale. 
Pour d\'emontrer l'inclusion inverse,  fixons $\lambd:=(\lambda_1,\ldots,\lambda_n) \in  {\rm Rel}_{\bf k}(f_1(\alpha),\ldots,f_n(\alpha))$. 
Posons \'egalement $\f(z) := (f_1(z),\ldots,f_n(z))$. 
En it\'erant le syst\`eme mahl\'erien, on trouve 
\begin{align*}
0 &= \langle \lambd \,, \, (f_1(\alpha),\ldots,f_n(\alpha)) \rangle \\
&=\left\langle \lambd \, ,\, A_l(\alpha)(f_1(\alpha^{q^l}),\ldots,f_n(\alpha^{q^l})) \right\rangle 
\\
&= \left\langle \lambd A_l(\alpha)\, , \, (f_1(\alpha^{q^l}),\ldots,f_n(\alpha^{q^l})) \right\rangle \,.
\end{align*}
Le point $\alpha^{q^l}$ \'etant par hypoth\`ese r\'egulier, le th\'eor\`eme \ref{thm: pphHomogene} implique que 
\begin{equation}\label{relationdecomposee1}
\lambd A_l(\alpha) \in {\rm ev}_{\alpha^{q^l}}\left({\rm Rel}_{ {\bf k}(z)}(f_1(z),\ldots,f_n(z))\right) \,.
\end{equation}
Nous allons montrer l'inclusion d'espaces vectoriels suivante : 
\begin{equation}
\label{inclusionespacerelations}
{\rm ev}_{\alpha^{q^l}}\left({\rm Rel}_{{\bf k}(z)}(\f(z))\right) \subset {\rm ev}_\alpha\left({\rm Rel}_{ {\bf k}(z)}(\f(z))\right) A_l(\alpha) \, .
\end{equation}
En effet, si $\bf{v}(z)\in {\rm Rel}_{{\bf k}(z)}(\f(z))$, alors, on a 
\begin{align*}
0&= \left\langle \v(z^{q^l}) , (f_1(z^{q^l},\ldots,f_n(z^{q^l}) \right\rangle \\
&= \left\langle \v(z^{q^l})A_l(z)^{-1} , \f(z) \right\rangle \,.
\end{align*}
On en d\'eduit que 
$ \v(z^{q^l}) A_l(z)^{-1} \in {\rm Rel}_{ {\bf k}(z)}(\f(z))$, 
ou encore que
$$ 
\v(z^{q^l}) \in  {\rm Rel}_{{\bf k}(z)}(\f(z)) A_l(z) \, .
$$
En \'evaluant en $\alpha$, il vient 
$$
\v(\alpha^{q^l}) = {\rm ev}_{\alpha^{q^l}}(\v(z))\in \ev_\alpha\left({\rm Rel}_{{\bf k}(z)}(\f(z))\right) A_l(\alpha)\, ,
$$ 
ce qui montre l'inclusion \eqref{inclusionespacerelations}.

En combinant \eqref{relationdecomposee1} et \eqref{inclusionespacerelations}, on trouve l'existence d'un vecteur 
${\boldsymbol \mu} \in \ev_\alpha\left({\rm Rel}_{ {\bf k}(z)}(\f(z))\right)$ tel que 
$$
\lambd A_l(\alpha) = {\boldsymbol \mu} A_l(\alpha) \, .
$$
Ainsi, on obtient bien le r\'esultat souhait\'e, \`a savoir :  $\lambd = (\lambd- {\boldsymbol \mu}) + {\boldsymbol \mu}$,  
o\`u $(\lambd-{\boldsymbol \mu}) \in {\rm ker}_{\bf k}A_l(\alpha)$ et ${\boldsymbol \mu} \in \ev_\alpha\left({\rm Rel}_{ {\bf k}(z)}(\f(z))\right)$. 
\end{proof}

Nous sommes \`a pr\'esent en mesure de prouver le th\'eor\`eme \ref{thm: baker}. 

\begin{proof}[D\'emonstration du th\'eor\`eme \ref{thm: baker}] 
Comme les fonctions $f_1(z),\ldots,f_n(z)$ sont $q$-mahl\'eriennes, on peut trouver un entier $m \geq n$, des fonctions $f_{n+1}(z),\ldots,f_{m}(z)$ 
et une matrice $A(z)\in {\rm GL}_m({\bf k}(z))$  tels que 
\begin{equation*}
\left( \begin{array}{ c }
     f_1(z) \\
     \vdots \\
     f_m(z)
  \end{array} \right) = A(z)\left( \begin{array}{ c }
     f_1(z^q) \\
     \vdots \\
     f_m(z^q)
  \end{array} \right)  \, .
\end{equation*}

D'apr\`es le th\'eor\`eme \ref{thm: paspole}, quitte \`a changer $q$ en $q^l$ et \`a modifier la matrice $A(z)$, 
on peut supposer que $\alpha$ n'est pas un p\^ole de $A(z)$ et que $\alpha^q$ est un point r\'egulier 
pour ce syst\`eme. \'Etant donn\'e un corps ${\bf k}_0$, on note 
$$
E_{{\bf k}_0} :=  \left\{(w_1,\ldots,w_m) \in {\bf k}_0^m\ \mid \ w_{n+1} = \cdots = w_m= 0\right\} \, .
$$
Ainsi, 
\begin{equation}\label{eq: pi}
{\rm Rel}_{\Q}(f_1,(\alpha),\ldots,f_n(\alpha))  =\pi\left({\rm Rel}_{\Q}(f_1,(\alpha),\ldots,f_m(\alpha)) \cap E_{\Q}\right)
\end{equation}
o\`u $\pi$ d\'esigne la projection des $n$ premi\`eres coordonn\'ees de  $\mathbb C^m$ dans $\mathbb C^n$.  
 Le th\'eor\`eme \ref{thm: StructureRelationsLineaires} implique alors l'\'egalit\'e suivante : 
$$
{\rm Rel}_{\Q}(f_1,(\alpha),\ldots,f_m(\alpha))
 =   {\rm ker}_{ \Q}A(\alpha) + {\rm ev}_\alpha\left({\rm Rel}_{ \Q(z)}(f_1(z),\ldots,f_m(z)\right)
 \, .
$$
Comme la matrice $A(z)$ est \`a coefficients dans ${\bf k}(z)$, on a : 
$$
{\rm ker}_{ \Q}A(\alpha) = {\rm Vect}_{ \Q} \left({\rm ker}_{\bf k}A(\alpha)\right) \, .
$$
D'autre part, les fonctions $f_i$ \'etant \`a coefficients dans ${\bf k}$, on a \'egalement que :  
$$
{\rm Rel}_{ \Q(z)}(f_1(z),\ldots,f_m(z)) =  {\rm Vect}_{ \Q(z)}  \left({\rm Rel}_{{\bf k}(z)}(f_1(z),\ldots,f_n(z)\right) \, .
$$
Enfin, on a $E_{\Q} = {\rm Vect}_{\Q} \, E_{\bf k}$. 
On a donc 
$$
{\rm Rel}_{\Q}(f_1,(\alpha),\ldots,f_m(\alpha)) \cap E_{\Q} = {\rm Vect}_{ \Q} \left({\rm Rel}_{\bf k}\left(f_1,(\alpha),\ldots,f_m(\alpha)\right)  \cap E_{\bf k} \right) \, 
$$
et l'\'egalit\'e (\ref{eq: pi})  entraine  alors que
$$
{\rm Rel}_{\Q}\left(f_1,(\alpha),\ldots,f_n(\alpha)\right) = {\rm Vect}_{ \Q} {\rm Rel}_{\bf k}\left(f_1,(\alpha),\ldots,f_n(\alpha)\right) \, ,
$$
comme souhait\'e.
\end{proof}

\section{Relations lin\'eaires entre solutions d'un syst\`eme mahl\'erien}\label{sec: relfonc}

Dans cette section, nous \'etudions les relations fonctionnelles  de d\'ependance lin\'eaire 
entre les solutions d'un syst\`eme mahl\'erien. 
 
Fixons, pour toute cette section, les notations suivantes. 
Soient  ${\bf k}$ un corps de nombres et 
$f_1(z),\ldots,f_n(z) \in {\bf k}\{z\}$ des solutions d'un syst\`eme mahl\'erien du type  (\ref{eq: systeme}).    
On notera \'egalement de la m\^eme fa\c con $\f(z):=(f_1(z),\ldots,f_n(z)) \in {\bf k}[[z]]^n$ et $\f(z) := \sum_{i=0}^\infty \f_i z^i\in{\bf k}^n[[z]]$, 
o\`u $\f_i$ d\'esigne le vecteur colonne de $\k^n$ dont les coordonn\'ees correspondent aux $i$-\`emes coefficients des fonctions 
$f_1(z),\ldots,f_n(z)$. On rappelle que ${\bf k}[z]_h$ d\'esigne l'ensemble des polyn\^omes \`a coefficients dans ${\bf k}$ de degr\'e au plus $h$.  
On notera aussi 
$$
{\rm Rel}_{{\bf k}[z]_h}(f_1(z),\ldots,f_n(z)) = ({\bf k}[z]_h)^n \cap {\rm Rel}_{{\bf k}(z)}(f_1(z),\ldots, f_n(z)) \, ,
$$ 
le $\bf k$-espace vectoriel des relations lin\'eaires polynomiales de hauteur inf\'erieure ou \'egale \`a $h$. 
\'Etant donn\'es deux entiers  $h$ et $M$, on d\'efinit la matrice 
\begin{equation*}
{\mathcal S}(h,M,{\bf f}) := \left(\begin{array}{cccccc} \f_0 & \f_1 & \cdots & \f_h 
& \cdots & \f_M \\ 0 & \f_0 & \ddots & \ddots & \ddots & \f_{M-1} \\ 
\vdots & \ddots & &&& \vdots \\ 0 & \cdots & 0 & \f_0 & \cdots & \f_{M-h} \end{array}\right) \, ,
\end{equation*}
o\`u $\f_i$ d\'esigne le vecteur colonne de $\k^n$ dont les coordonn\'ees correspondent aux $i$-\`emes coefficients des fonctions 
$f_1(z),\ldots,f_n(z)$. On pose 
$$
{\rm ker}_{\bf k} {\mathcal S}(h,M,\f) := \left\{ { \lambd} \in ({\bf k}^n)^{h+1} \mid { \lambd} {\mathcal S}(h,M,\f) = 0\right\} \, .
$$
On d\'efinit \'egalement
l'isomorphisme $\varphi_h : (\k^n)^{h+1} \rightarrow (\k[z]_h)^n$ par 
$$
\varphi_h(\w_0,\ldots,\w_h) = \sum_{i=0}^h \w_i z^i \, .
$$
On omettra la d\'ependance en $h$ dans la suite, le notant simplement $\varphi$. 
On obtient alors le r\'esultat suivant. 

\begin{theo}
\label{BaseRelationsLin\'eaires1} Soient ${\bf k}$ un corps de nombres et $f_1(z),\ldots,f_n(z)\in {\bf k}\{z\}$  des fonctions solutions 
d'un syst\`eme du type (\ref{eq: systeme}).   
On note $b(z)$ le plus petit commun multiple des d\'enominateurs des coefficients de $A(z)$, 
$d$ le maximum des degr\'es des coefficients de la matrice $b(z)A(z)$, et $\nu$ la valuation en $0$ du polyn\^ome $b(z)$.  
Soient $h:=4^nd$ et 
$$
c:= \left\lceil \frac{q^{n \left(\frac{qh+d+1}{q-1} + q + 1 \right)}(h+q) + \nu - \frac{h+d}{q-1} }{q-1} \right\rceil \, \cdot
$$  
On a : 
\begin{eqnarray*}
{\rm Rel}_{{\bf k}(z)}(f_1(z),\ldots,f_n(z)) &=& {\rm Vect}_{{\bf k}(z)} \left({\rm Rel}_{{\bf k}[z]_h}(f_1(z),\ldots f_n(z)\right)\\
& =&  {\rm Vect}_{{\bf k}(z)}\left(\varphi \left(\ker_{\bf k} {\mathcal S}(h,c, {\bf f})\right) \right)\, .
\end{eqnarray*} 
\end{theo}

D\'eterminer si les fonctions $f_1(z),\ldots,f_n(z)$ sont lin\'eairement ind\'ependantes peut se faire de fa\c con plus efficace comme 
le montre le r\'esultat suivant. 

\begin{theo}
\label{BaseRelationsLin\'eaires2} 
Conservons les notations du th\'eor\`eme pr\'ec\'edent. 
Soit $h:=\left\lfloor d/(q-1) \right\rfloor$.  
On a les \'equivalences suivantes :  

\medskip

\begin{itemize}
\item[{\rm (i)}] ${\rm Rel}_{{\bf k}(z)}(f_1(z),\ldots,f_n(z)) \neq \{0\}$,

\medskip

 \item[{\rm (ii)}] ${\rm Rel}_{{\bf k}[z]_h}(f_1(z),\ldots,f_n(z)) \neq \{0\}$,
 
 \medskip

\item[{\rm (iii)}] $\ker_{\bf k} {\mathcal S}(h,c,{\bf f}) \neq \{0\}$. 
\end{itemize}
\end{theo}

Le reste de cette section est consacr\'ee aux d\'emonstrations des  th\'eor\`emes \ref{BaseRelationsLin\'eaires1} et 
\ref{BaseRelationsLin\'eaires2}. 
On montre d'abord le lemme de z\'ero suivant.

\begin{lem}
\label{degreannulationbase}
Soit $\w(z)\in ({\bf k}[z]_h)^n$.  
Si la valuation en $z=0$ de la s\'erie $\langle \w(z) , \f(z) \rangle$ est strictement sup\'erieure \`a 
\begin{equation}
\label{eq:degreannulationbase}
\left\lceil \frac{q^{n \left(\frac{qh+d+1}{q-1} + q + 1 \right)}(h+q) + \nu - \frac{h+d}{q-1} }{q-1} \right\rceil 
\end{equation} 
alors
$$
\langle \w(z) , \f(z) \rangle = 0 \, .
$$
\end{lem}

\begin{proof}
 Choisissons $\w(z)$ comme dans l'\'enonc\'e, et supposons par l'absurde que la valuation en $z=0$ de la s\'erie 
 $\langle \w(z) , \f(z) \rangle$ est finie et \'egale \`a
\begin{equation} \label{ordreannulation}
 M > \left\lceil \frac{q^{n \left(\frac{qh+d+1}{q-1} + q + 1 \right)}(h+q) + \nu - \frac{h+d}{q-1} }{q-1} \right\rceil  \,.
\end{equation} 
On montre tout d'abord qu'on peut se ramener au cas o\`u $b(z) = z^\nu$. En effet, posons 
$$
\beta(z) := b(z)z^{-\nu} \qquad \text{ et } \qquad u(z) := \prod_{i=0}^\infty \beta(z^{q^i}) \, .
$$
Il vient $u(z) = z^{-\nu}b(z)u(z^q)$ et, en posant $\tilde\f (z) := u(z)\f(z)$, on v\'erifie que $\tilde\f (z)$ est solution du syst\`eme  
\begin{equation}
\label{evitementpoles}
\tilde\f (z) =  z^{-\nu}\tilde A(z) \tilde \f(z^q),
\end{equation}
o\`u $\tilde A(z) = b(z)A(z)$ est une matrice \`a coefficients dans $\k[z]$. 
Comme la valuation en $z=0$ de la s\'erie $u$ est nulle, 
celle de $ \langle \w(z) , \tilde \f(z) \rangle$ est \'egale \`a $M$.   
Dans la suite, nous supposerons donc que
$$
A(z) = z^{-\nu} \tilde A(z), \; \mbox{ avec } \tilde A(z) \in \mathcal M_{n}({\bf k}[z])\,.
$$
Pour tout entier $i$, on consid\`ere le vecteur suivant : 
$$
\g_i := \left(\begin{array}{c} \f_i \\ \vdots \\ \f_{i - h}\end{array}\right) \in (\k^n)^{h+1} \,.
$$
On pose \'egalement $\v := \varphi^{-1}(\w(z))$. 
L'\'egalit\'e \eqref{ordreannulation} se r\'ecrit : 
\begin{equation} \label{ordreannulation1} \left\{
\begin{array}{llll} 
\langle \v , \g_i \rangle&  = & 0 & \forall i < M \, , \\
 \langle \v , \g_M \rangle&  \neq & 0\, . & \end{array} 
\right.
 \end{equation}
Notons $\g(z) := \sum \g_i z^i$ et consid\'erons la matrice par blocs suivante :
$$
B(z) := \left(\begin{array}{cccc} \tilde A(z) & 0 & \cdots &0 \\ z\tilde A(z)  & \vdots 
& & \vdots \\ \vdots & \vdots & \ddots & \vdots \\ z^h\tilde A(z) & 0 & \cdots & 0 \end{array} \right)  \in \mathcal M_{n(h+1)}({\bf k}[z])\, .
$$
On obtient ainsi une nouvelle \'equation mahl\'erienne : 
\begin{equation}
\label{equationmahleriennegrandit}
\g(z) = z^{-\nu}B(z)\g(z^k) \, .
\end{equation}
Remarquons au passage que 
$\deg B \leq \deg \tilde A + h \leq d + h$.   
On peut alors d\'ecomposer la matrice $B(z)$ selon les puissances de $z$
$$
B(z) = \sum_{i=0}^{\deg B} B_i z^i \, ,
$$
o\`u les matrices $B_i$ appartiennent \`a $\mathcal M_{n(h+1)}({\bf k})$. 
Notons alors
$$
e := \left\lfloor \frac{\deg B(z)}{q} \right\rfloor + 1 \,.
$$
Soit $i$ un entier. On consid\`ere l'unique couple d'entiers $(s,j)$ tel que $i = qs + j -\nu$ et $0 \leq j < q$. 
L'identification des termes en $z^i$ dans l'\'equation \eqref{equationmahleriennegrandit} induit la relation suivante : 
\begin{equation}
\label{recurrencecoeffmahlerien}
\g_i = \g_{qs+j-\nu} = \sum_{r=0}^{e-1} B_{qr+j} \g_{s-r} \, ,
\end{equation}
o\`u l'on a pos\'e $B_l= 0$ pour $l > \deg B$, et $\g_l = 0$ pour $l < 0$. 
Fixons $m>0$ et
choisissons alors $i'$ un entier tel que $i\leq i' < i+qm$. Si $i' = qs' + j' - \nu$, avec $j' < q$, on a n\'ecessairement $s \leq s' \leq qm$. On a \'egalement 
\begin{equation}
\label{recurrencecoeffmahlerien2}
\g_{i'}  =  \sum_{r=0}^{e-1} B_{qr+j'} \g_{s'-r} \,.
\end{equation}
Le couple $(i,m)$ \'etant fix\'e, Les vecteurs $\g_l$ intervenant dans le membre de droite sont donc les 
$e+m+1$ vecteurs $\g_{s-e+1},\cdots,\g_{s+m+1}$. 
 D\'efinissons alors l'entier $m$ de la mani\`ere suivante : 
$$ 
m := \left \lfloor \frac{e}{q-1} \right\rfloor + 1 \,.
$$ 
Il est choisi de sorte que $qm \geq e+m+1$. On consid\`ere alors, pour chaque entier $i$, 
le vecteur  
\begin{equation}
\label{vecteurcolonne}
{\bf G}_{i}= \left(\begin{array}{c}
\g_{i+qm-1} \\ 
\vdots \\ 
\g_{i} \end{array}\right)
 \in \k^{n(h+1)qm} \, .
\end{equation}
Remarquons que par d\'efinition des $\g_i$, on a 
$$
\G_i: = \left(\begin{array}{c} \f_{i+qm-1}\\ \f_{i+qm-2} \\ \vdots \\ \f_{i+qm-h} \\ \f_{i+qm-2} \\ \vdots \\ \f_{i-h+1} \end{array}\right) \,.
$$
Les vecteurs $\G_i$ appartiennent donc tous au sous-espace vectoriel $W$ de $\k^{n(h+1)qm}$ form\'e des vecteurs de la forme 
$$
(X_1,X_2,\cdots,X_{h},X_2,X_3,\cdots,X_{h+1},X_3,\cdots,X_{qm+h})^T, \; \text{ o\`u } X_l \in \k^n \,.
$$
L'identit\'e \eqref{recurrencecoeffmahlerien2} assure l'existence de $q$ matrices $C_0,\ldots,C_{q-1} \in \mathcal{M}_{n(h+1)qm}(\k)$ telles que, 
pour tout entier $i$, on a : 
\begin{equation}
\label{recurrenceG} \G_{i}=C_j\G_{s-e+1} \, ,
\end{equation}
o\`u $i=qs+j-\nu$. On consid\`ere \`a pr\'esent la suite d'espaces embo\^{i}t\'es
$$
V_l := {\rm Vect}\left\{\G_0,\ldots,\G_l\right\} \subset W \, .
$$
Pour conclure cette d\'emonstration on aura besoin du lemme suivant.

\begin{lem}\label{stationnarite}
Soit $a \geq (\nu - d)/(q- 1)$ un entier. Si 
$$
V_a=V_{q(a+e)-\nu-1} \,,
$$
alors la suite $(V_l)_{l \geq 0}$ est constante \`a partir du rang $a$.
\end{lem}

\begin{proof}[D\'emonstration du lemme \ref{stationnarite}]
Choisissons un entier $j< q$. Pour tout $r$, l'\'egalit\'e \eqref{recurrenceG} implique que 
$$
C_j \G_r= \G_{q(r+e-1)+j-\nu} \,.
$$
Si $r \leq a$, alors $q(r+e-1)+j-\nu \leq q(a+e)-\nu-1$ et donc 
$$
C_j \G_r \in V_{q(a+e)-\nu-1} = V_a \, .
$$
L'action du mono\"ide engendr\'e par les matrices $C_0,\ldots,C_{q-1}$ stabilise donc l'espace $V_a$. 
Prenons maintenant un entier $b > a$. Comme $a \geq (\nu - d)/(q- 1)$, l'\'egalit\'e \eqref{recurrenceG} assure l'existence d'indices 
$j_1,\ldots,j_r$ et d'un entier $l \leq a$ tels que 
$$
\G_b = C_{j_1}\cdots C_{j_r} \G_l \,.
$$
Donc $\G_b \in V_a$, ce qui termine la d\'emonstration.
\end{proof}

Nous sommes \`a pr\'esent en mesure d'achever la preuve du lemme \ref{degreannulationbase}. 
Les vecteurs $\G_i$ appartiennent tous \`a l'espace $W$. Un calcul rapide montre que la dimension de cet espace n'exc\`ede pas
$n(qm+h)$. Partant de $a =\lfloor (\nu - d)(q-1)^{-1} \rfloor$, et en appliquant de fa\c con r\'ecursive le lemme \ref{stationnarite}, 
on obtient que la suite d'espaces vectoriels embo\^{i}t\'es $V_l$ est n\'ecessairement constante \`a partir du rang
\begin{equation}
\label{rangstation}
\left\lceil \frac{q^{n(h+qm)}(qe-d) + \nu- qe}{q - 1}\right \rceil \,.
\end{equation}
En rempla\c cant $e$ et $m$ par leurs valeurs respectives, on obtient que 
\begin{eqnarray*}
M &> &\left\lceil \frac{q^{n \left(\frac{qh+d+1}{q-1} + q + 1 \right)}(h+q) + \nu - \frac{h+d}{q-1} }{q-1} \right\rceil  \\ 
&\geq& \left\lceil \frac{q^{n(h+qm)}(qe-d) + \nu- qe}{q - 1} \right \rceil\, .
\end{eqnarray*}
La suite $(V_l)_{l \in \N}$ est donc constante \`a partir du rang $M-1$ et 
donc $\G_M \in V_{M-1}$. 

D'apr\`es \eqref{ordreannulation1}, le vecteur $\v=(\w_0,\ldots,\w_h) \in (\k^n)^{h+1}$ est tel que 
$$
\langle \v , \g_i \rangle = 0 \,,
$$
pour tout entier $i < M$. 
Posons ${\bf u} :=(\v,\cdots,\v) \in \k^{n(h+1) qm}$.  Il vient alors que
$$
\langle \bf u , \G_i \rangle = 0 \, ,
$$
 pour tout $i < M$. 
Le vecteur $\bf u$ appartient donc \`a l'orthogonal de $V_{M-1}$. Mais comme $\G_M\in  V_{M-1}$, le vecteur $\bf u$ est aussi orthogonal au vecteur 
$\G_M$ et donc $\langle {\bf u} , \G_M \rangle = 0$. En consid\'erant les $n(h+1)$ derni\`eres coordonn\'ees des vecteurs $\G_M$ et $\bf u$, on voit que cette 
derni\`ere \'egalit\'e est \'equivalente \`a 
$$
\langle \v , \g_M \rangle = 0 \,,
$$
ce qui contredit \eqref{ordreannulation1}. Cela termine la d\'emonstration. 
\end{proof}

Le lemme suivant assure l'existence, lorsque les fonctions $f_1(z),\ldots,f_n(z)$ sont lin\'eairement d\'ependantes sur ${\bf k}(z)$, 
 d'une relation de d\'ependance lin\'eaire de petite hauteur.  

\begin{lem}
\label{DependanceLineaire}
Supposons que les fonctions $f_1(z),\ldots,f_n(z)$ sont lin\'eairement d\'ependantes sur ${\bf k}(z)$. Soit $h:= \lfloor d/(q-1)\rfloor$.  Alors, il existe un vecteur non nul 
$\w(z) :=(w_0(z),\ldots,w_n(z)) \in ({\bf k}[z]_h)^n$  tel que
$$
\sum_{i=1}^n w_i(z)f_i(z)=0 \,.
$$
\end{lem}

\begin{proof}
Soit $h$ le degr\'e de la plus petite relation lin\'eaire non triviale sur ${\bf k}[z]$ entre les fonctions $f_1(z),\ldots,f_n(z)$, et 
$\w(z)$ le vecteur des coefficients d'une telle relation. Utilisant la relation fonctionnelle \eqref{eq: systeme}, 
on \'ecrit
\begin{eqnarray*}
0 & = & \langle \w(z),\f(z) \rangle 
\\ & = &\langle \w(z),\beta(z)^{-1}\tilde A(z)\f(z^q) \rangle 
\\ & = &  \langle \w(z)\tilde A(z),\f(z^q) \rangle \, .
\end{eqnarray*}
On peut d\'ecomposer le vecteur $\w(z) \tilde A(z) \in \k[z]^n$ selon les restes de puissances de $z$ modulo $q$. 
On obtient alors la d\'ecomposition unique suivante : 
$$
\w(z)\tilde A(z) = \sum_{i=0}^{q-1}z^i \v_i(z^q) \,.
$$
L'identit\'e $\langle \w(z)\tilde A(z),\f(z^q) \rangle =0$ implique alors que  
$$
\langle \v_i(z),\f(z)\rangle = 0 \,,
$$
pour tout $i$, $0\leq i \leq q-1$. 
Comme par hypoth\`ese $\w(z)$ est non nul, il existe un indice $i_0$  tel que le vecteur $\v_{i_0}(z)$ soit non nul. 
Notons $l$ le degr\'e maximal des coefficients de $\v_{i_0}$.  
Par minimalit\'e de $h$, on a $l \geq h$. D'autre part, comme le degr\'e de $\w(z) \tilde A(z)$ est inf\'erieur \`a $h + d$, 
on a
$$
ql + i_0 \leq h + d \, .
$$
On en d\'eduit que $qh\leq h +d$, ce qui permet de conclure. 
\end{proof}

Nous sommes \`a pr\'esent en mesure de prouver les th\'eor\`emes \ref{BaseRelationsLin\'eaires1} et \ref{BaseRelationsLin\'eaires2}. 

\begin{proof}[D\'emonstration du th\'eor\`eme \ref{BaseRelationsLin\'eaires1}]
Posons $h:= 4^nd$ et 
$$
c := \left\lceil \frac{q^{n \left(\frac{qh+d+1}{q-1} + q + 1 \right)}(h+q) + \nu - \frac{h+d}{q-1} }{q-1} \right\rceil  \, \cdot
$$
Les inclusions  
$$
 {\rm Vect}_{\k(z)} \left({\rm Rel}_{{\bf k}[z]_h}(f_1(z),\ldots,f_n(z)\right) \subset{\rm Rel}_{\k(z)}(f_1(z),\ldots,f_n(z))
$$ 
et 
$$
{\rm Rel}_{{\bf k}[z]_h}(f_1(z),\ldots,f_n(z))  \subset \varphi \left(\ker_{\bf k} {\mathcal S}(h,c,\f)\right)
$$  
sont imm\'ediates. 
L'inclusion $ \varphi \left(\ker_{\bf k} {\mathcal S}(h,c,\f)\right)\subset  {\rm Rel}_{{\bf k}[z]_h}(f_1(z),\ldots,f_n(z))$ 
est quant \`a elle une reformulation directe du lemme \ref{degreannulationbase}. Il ne reste donc plus qu'\`a prouver l'inclusion
$$
{\rm Rel}_{\k(z)}(f_1(z),\ldots,f_n(z)) \subset   {\rm Vect}_{\k(z)} \left({\rm Rel}_{{\bf k}[z]_h}(f_1(z),\ldots,f_n(z)\right) \,.
$$
Montrer cette inclusion revient \`a montrer l'existence d'une base de relations lin\'eaires dont les coefficients sont des polyn\^omes 
de degr\'e inf\'erieur \`a $4^nd$. Soit $r$ la dimension de l'espace ${\rm Rel}_{\k(z)}(f_1(z),\ldots,f_n(z))$. 
Nous allons montrer le r\'esultat par r\'ecurrence sur l'entier $r$.

Si $r=0$, il n'y a rien \`a prouver. Supposons alors l'\'egalit\'e vraie pour tout syst\`eme de la forme \eqref{eq: systeme} lorsque 
la dimension de l'espace des relations lin\'eaires entre les fonctions vaut $r-1$. 
Supposons que 
$$
\dim  {\rm Rel}_{\k(z)}(f_1(z),\ldots,f_n(z)) = r \,.
$$
On va r\'eduire le syst\`eme pour se ramener \`a un syst\`eme de taille $n-1$. 
D'apr\`es le lemme \ref{DependanceLineaire}, il existe un vecteur $\w(z) \in (\k[z]_{h_0})^n$, o\`u $h_0=\left\lfloor \frac{d}{q-1} \right\rfloor$, 
tel que 
\begin{equation}
\label{relationpluspetitdegre}
\langle \w(z) , \f(z)  \rangle = w_1(z)f_1(z) + \cdots + w_n(z)f_n(z) = 0 \,.
\end{equation}
On peut supposer, quitte \`a renum\'eroter les fonctions, que $w_n(z) \neq 0$. 
On conjugue alors l'\'equation fonctionnelle de $\f$ par la matrice
$$
S(z) = 
\left(\begin{array}{ccccc} 1 & 0 \cdots & \cdots & 0 \\ 0 & 1 & \ddots &  & \vdots \\ \vdots 
& \ddots & \ddots & \ddots & \vdots \\ 0 & \cdots & 0& 1 & 0 \\ w_1(z) & w_2(z) 
& \cdots & w_{n-1}(z) & w_n(z) \end{array} \right) \,.
$$
On obtient le syst\`eme
$$
\left(\begin{array}{c} f_1(z) \\ \vdots \\ f_{n-1}(z) \\ 0 \end{array}\right)
= S(z)A(z)S(z^q)^{-1} \left(\begin{array}{c} f_1(z^q) \\ \vdots \\ f_{n-1}(z^q) \\ 0 \end{array}\right) \,.
$$
Notons $B(z)$ le mineur principal de taille $(n-1)$ de la matrice $S(z)A(z)S(z^q)^{-1}$. On obtient alors le syst\`eme de taille $(n-1)$ suivant : 
\begin{equation}
\label{systemer\'eduit}
 \left(\begin{array}{c} f_1(z) \\ \vdots \\ f_{n-1}(z) \end{array}\right)=B(z) \left(\begin{array}{c} f_1(z^q) \\ \vdots \\ f_{n-1}(z^q)\end{array}\right) \, .
\end{equation}
Par construction, le polyn\^ome $w_n(z^q)b(z)$ est un multiple commun aux d\'enominateurs des coefficients de la matrice $B(z)$. 
La matrice $w_n(z^q)b(z)B(z)$ est donc \`a coefficients polynomiaux. On peut majorer le degr\'e de ces polyn\^{o}mes par 
$$ 
h_0(q+1) +d \leq \frac{d}{q-1}(q+1) + d \leq d\frac{2q}{q-1} \leq 4d \,.
$$
Comme $w_n(z)\neq 0$, on a 
$$
\dim {\rm Rel}_{\k(z)}(f_1(z),\ldots,f_{n-1}(z)) = r-1
$$ 
et on peut donc appliquer l'hypoth\`ese de r\'ecurrence au syst\`eme (\ref{systemer\'eduit}). 
On obtient : 
$$
{\rm Rel}_{\k(z)}(f_1(z),\ldots,f_{n-1}(z)) \subset {\rm Vect}_{\k(z)} \left( {\rm Rel}_{{\bf k}[z]_{h_1}}(f_1(z),\ldots,f_{n-1}(z))\right) \,,
$$ 
o\`u $h_1 := 4^{n-1}4d= 4^nd=h$.  
D'autre part, on a  
$$
{\rm Rel}_{\k(z)}(f_1(z),\ldots,f_{n}(z)) = {\rm Rel}_{\k(z)}(f_1(z),\ldots,f_{n-1}(z))  
\oplus_{\k(z)} \k(z).\w(z) \,,
$$ 
et par hypoth\`ese 
$\w(z) \in {\rm Rel}_{{\bf k}[z]_h}(f_1(z),\ldots,f_n(z))$. 
On en d\'eduit 
$$ 
{\rm Rel}_{\k(z)}(f_1(z),\ldots,f_n(z)) \subset   {\rm Vect}_{\k(z)} \left({\rm Rel}_{{\bf k}[z]_h}(f_1(z),\ldots,f_n(z))\right) \,,
$$
comme voulu. Cela conclut la d\'emonstration. 
\end{proof}

\begin{proof}[D\'emonstration du th\'eor\`eme \ref{BaseRelationsLin\'eaires2}]
Le th\'eor\`eme d\'ecoule imm\'ediatement des lemmes  \ref{degreannulationbase} et \ref{DependanceLineaire}. 
\end{proof}

\section{Deux exemples de syst\`emes automatiques}\label{sec: ex}

Nous illustrons ici les r\'esultats obtenus \`a travers deux exemples. Pour les d\'efinitions relatives aux automates finis et aux suites automatiques, nous renvoyons le lecteur 
\`a \cite{AS}. 

\subsection{Premier exemple}

On d\'efinit la suite binaire $(a_n)_{n\geq 0}$ de la fa\c con suivante : $a_n=0$ si le d\'eveloppement en base $3$ de l'entier $n$ a un nombre pair de 
chiffres \'egaux \`a $2$ 
et $a_n=1$ si ce nombre est impair. Il s'agit d'une variante de la c\'el\`ebre suite de Thue--Morse. Notons $f(z) := \sum_{n\geq 0} a_n z^n \in \mathbb Q\{z\}$ 
la s\'erie g\'en\'eratrice associ\'ee ; elle est analytique dans le disque unit\'e ouvert. 
Par d\'efinition, il s'agit d'une s\'erie $3$-automatique. Comme la suite $(a_n)_{n\geq 0}$ ne prend qu'un nombre fini de valeurs enti\`eres et qu'elle 
n'est pas ultimement p\'eriodique, on obtient facilement que $f(z)$ est transcendante. 
Pour $\alpha$ alg\'ebrique, $0<\vert \alpha\vert <1$, on se propose d'\'etudier 
la transcendance du 
nombre automatique $f(\alpha)$. 


\begin{figure}[htbp]
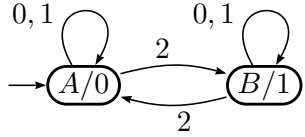

\centering
\VCDraw{%
        \begin{VCPicture}{(0,-1)(4,2)}
 \StateVar[A/0]{(0,0)}{a}  \StateVar[B/1]{(4,0)}{b}
\Initial[w]{a}
\LoopN{a}{0,1}
\LoopN{b}{0,1}
\ArcL{a}{b}{2}
\ArcL{b}{a}{2}
\end{VCPicture}
}
\caption{Un $3$-automate produisant la suite $(a_n)_{n\geq 0}$.}
  \label{AB:figure:thue}
\end{figure}

Posons $f_1(z):=f(z)$ et $f_2(z) := \sum_{n\geq 0} (1-a_n)z^n$. On v\'erifie sans peine que ces deux fonctions sont $3$-mahl\'eriennes 
et solutions du syst\`eme fonctionnel suivant :
\begin{equation*}
\left( \begin{array}{ c }
     f_1(z) \\
     f_2(z)\\
  \end{array} \right) = A(z)\left( \begin{array}{ c }
     f_1(z^3) \\
     f_2(z^3)\\
  \end{array} \right) \, ,
\end{equation*}
o\`u
$$
A(z) := \left(\begin{array}{cc} 
1+z & z^2 \\

z^2 & 1+ z
\end{array}\right)  \, .
$$
Le th\'eor\`eme \ref{BaseRelationsLin\'eaires2} permet de montrer facilement que les fonctions $f_1(z)$ et $f_2(z)$ sont lin\'eairement ind\'ependantes sur $\mathbb C(z)$. 
Avec les notations du th\'eor\`eme \ref{BaseRelationsLin\'eaires2}, on a 
ici $n=2$, $d=2$, $q=3$ et $\nu=0$. On en d\'eduit que  $h=1$ et $c=(3^{14} \times 8- 3)/4$.  On v\'erifie alors que les quatre premi\`eres colonnes 
$$\left(\begin{array}{c} 0 \\ 1 \\ 0 \\ 0 \end{array} \right),  
\left(\begin{array}{c}
 0 \\ 1 \\ 0 \\ 1 \end{array} \right),  
\left(\begin{array}{c} 1 \\ 0 \\ 0 \\ 1 \end{array} \right) \text{ et }
 \left(\begin{array}{c} 0 \\ 1 \\ 1 \\ 0 \end{array} \right)
 $$
de la matrice ${\mathcal S}(1,(3^{14}\times 8-3)/4,\f)$ sont lin\'eairement ind\'ependantes, o\`u $\f(z) := (f_1(z),f_2(z))$. 
Il suit que  $\ker_{\mathbb Q}{\mathcal S}(1,(3^{14}\times 8-3)/4,\f)=\{0\}$, 
ce qui permet de conclure. On a cependant par d\'efinition la relation affine suivante :
$$
f_1(z)+f_2(z) = \frac{1}{1-z} \, \cdot
$$

Le d\'eterminant de $A(z)$ n'a qu'une racine dans le disque unit\'e ouvert ; il s'agit du point $\phi:= (1-\sqrt 5)/2$.  
On obtient donc que l'ensemble des points singuliers de notre syst\`eme est 
$$
{\mathcal E} := \left\{ \phi^{1/3^{l}} \mid l\geq 1\right\} \, .
$$
En un point alg\'ebrique $\alpha$ qui n'est pas dans $\mathcal E$, le th\'eor\`eme \ref{thm: pphHomogene} implique directement que $f(\alpha)$ est transcendant. En effet, dans 
le cas contraire on aurait \'egalement $f_2(\alpha)$ alg\'ebrique, car  $f_1(z)+f_2(z)=1/(1-z)$, et on en d\'eduirait donc  une relation lin\'eaire sur $\Q$ entre 
$f_1(\alpha)$ et $f_2(\alpha)$, ce qui est impossible puisque $f_1(z)$ et $f_2(z)$ sont lin\'eairement ind\'ependantes sur $\Q(z)$.   
Notons qu'au point $\phi$, on a 
$$
\left( \begin{array}{ c }
     f_1(\phi) \\
     f_2(\phi)\\
  \end{array} \right) = \left(\begin{array}{cc} 
1+\phi & 1+ \phi \\

1+ \phi & 1+ \phi
\end{array}\right)\left( \begin{array}{ c }
     f_1(\phi^3) \\
     f_2(\phi^3)\\
  \end{array} \right) \,
$$ 
et en particulier $f_1(\phi)=f_2(\phi)$.  Comme $f_1(z)+f_2(z)=1/(1-z)$, on en d\'eduit que $f_1(\phi)= -\phi/2 \in \mathbb Q(\phi)$, en d\'epit du fait que les coefficients de 
$f_1(z)$ soient des entiers et que $f_1(z)$ soit transcendante.   
En raisonnant par r\'ecurrence, on obtient d'ailleurs que $f_1(\phi^{1/3^l}) \in \mathbb Q\left(\phi^{1/3l}\right)$ pour tout entier $l\geq 1$.  
En r\'esum\'e, on obtient  le r\'esultat suivant.

\begin{prop}
Soit $\alpha$ un nombre alg\'ebrique, $0<\vert \alpha\vert <1$. On a l'alternative suivante : 
soit il existe $l\geq 1$ tel que $\alpha = \phi^{1/3^{l}}$ et alors $f(\alpha)\in \mathbb Q\left(\phi^{1/3^{l}}\right)$, 
soit $f(\alpha)$ est transcendant. 
\end{prop}

Comme les fonctions $f_1(z)$ et $f_2(z)$ sont lin\'eairement ind\'ependantes sur $\Q(z)$, les relations de d\'ependance lin\'eaire entre leurs valeurs 
apparaissent donc toutes comme ayant une origine matricielle (elles sont donn\'ees par le noyau des matrices $A_l(\phi)$).  
Cette origine des relations d\'epend en fait du syst\`eme choisi  pour les \'etudier et on assiste \`a un principe des vases  communicants, comme l'illustre la remarque suivante.  
En appliquant l'astuce de d\'edoublement du lemme \ref{lem: dedoublement}, on peut obtenir le syst\`eme suivant :
\begin{equation*}
\left( \begin{array}{ c }
     f_1(z) \\
     f_2(z)\\
     f_1(z^3)\\
     f_2(z^3)
  \end{array} \right) = B(z)\left( \begin{array}{ c }
     f_1(z^3) \\
     f_2(z^3)\\
     f_1(z^9)\\
     f_2(z^9)
  \end{array} \right) \, ,
\end{equation*}
avec 
$$
B(z) := \left(\begin{array}{cccc} 
z & z^2 & 1+z^3& z^6\\

z^2 &  z & z^6& 1+z^3\\

1&0&0&0\\
0&1&0&0
\end{array}\right)  \, .
$$
On peut v\'erifier que $\phi$ est  \`a pr\'esent un point r\'egulier 
 pour ce nouveau syst\`eme. Evidemment, le d\'edoublement du syst\`eme a cr\'e\'e deux relations lin\'eaires libres entre les fonctions 
 $f_1(z),f_2(z),f_3(z)$ et $f_4(z)$, 
\`a savoir,
$$
f_1(z) - (1+z)f_3(z) - z^2f_4(z) = 0 \; \text{ et }\; f_2(z) - z^2f_3(z) - (1+z)f_4(z) \,.
$$
Les fonctions $f_1(z), f_2(z), f_1(z^3)$ et $f_2(z^3)$ sont donc \`a pr\'esent lin\'eairement d\'ependantes 
sur $\mathbb Q(z)$ et il n'y a pas de contradiction. Les relations lin\'eaires entre les fonctions $f_1(z)$ et $f_2(z)$ aux points de l'ensemble $\mathcal E$ 
apparaissent dans ce nouveau syst\`eme  
comme ayant une origine fonctionnelle, c'est-\`a-dire qu'elles sont obtenues comme sp\'ecialisation d'une relation 
sur $\Q(z)$ entre les fonctions  $f_1(z),f_2(z),f_3(z)$ et $f_4(z)$.

\subsection{Second exemple}

Il s'agit d'une variante de l'exemple pr\'ec\'edent. 
On consid\`ere les quatre suites binaires ${\bf a}_1 := (a_{1,n})_{n\geq 0}, {\bf a}_2 :=(a_{2,n})_{n\geq 0}, {\bf a}_3 :=(a_{3,n})_{n\geq 0}$ et ${\bf a}_4 :=(a_{4,n})_{n\geq 0}$ 
d\'efinies comme suit :  
\begin{equation*}
\left\{ \begin{array}{rcl} a_{1,n} & = & 1 \iff (n)_3 \text{ a un nombre pair de $1$ et de $2$}
\\ a_{2,n} & = & 1\iff (n)_3 \text{ a un nombre impair de $1$ et pair de $2$}
\\ a_{3,n} & = & 1 \iff (n)_3 \text{ a un nombre impair de $1$ et de $2$}
\\ a_{4,n} & = & 1 \iff (n)_3 \text{ a un nombre pair de $1$ et impair de $2$} \end{array} \right.
\end{equation*}
On consid\`ere \'egalement les s\'eries g\'en\'eratrices  associ\'ees \`a ces suites, \`a savoir :  
$$
g_i(z) := \sum_{n \geq 0} a_{i,n} z^n, \;\;  1\leq i\leq 4 \, .
$$
Les fonctions $g_i(z)$ sont  $3$-automatiques. Plus pr\'ecis\'ement, \'etant donn\'es des nombres alg\'ebriques $\omega_1,\omega_2,\omega_3,\omega_4$, 
on peut v\'erifier que la suite  $\omega_1{\bf a}_1 + \omega_2{\bf a}_2 +\omega_3{\bf a}_3 + \omega_4{\bf a}_4$ 
est engendr\'e par le $3$-automate suivant : 

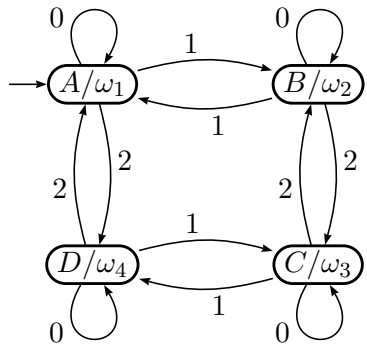
\begin{figure}[h]
\centering 
\VCDraw{  \begin{VCPicture}{(0,-5)(4,2)}
    \StateVar[A/\omega_1]{(0,0)}{A} \StateVar[B/\omega_2]{(5,0)}{B}   \StateVar[D/\omega_4]{(0,-4)}{C}    \StateVar[C/\omega_3]{(5,-4)}{D}  
    \Initial{A} 
     \ArcL{A}{B}{1}   \ArcL{B}{A}{1}   \ArcL{D}{C}{1}   \ArcL{C}{D}{1} \ArcL{A}{C}{2}  \ArcL{C}{A}{2}  \ArcL{B}{D}{2}  \ArcL{D}{B}{2}   \LoopN{A}{0} 
    \LoopN{B}{0}  \LoopS{C}{0}   \LoopS{D}{0} 
  \end{VCPicture}}
  \label{AB:figure:3n3m}
\end{figure}

\medskip

\noindent  \'Etant donn\'e un nombre alg\'ebrique $\alpha$, $0<\vert \alpha\vert <1$, on se propose de d\'ecrire les vecteurs 
$(\omega_1,\omega_2,\omega_3,\omega_4)\in \Q^4$ pour lesquels le nombre \og automatique\fg{} 
$$
\omega_1g_1(\alpha) + \omega_2g_2(\alpha) +\omega_3 g_3(\alpha)  + \omega_4 g_4(\alpha)
$$
est transcendant.  

Les fonctions $g_1(z),g_2(z),g_3(z)$ et $g_4(z)$ sont solutions du syst\`eme  suivant : 
\begin{equation*}
\left(\begin{array}{c} g_1(z) \\ g_2(z)\\ g_3(z) \\ g_4(z) \end{array} \right) = \left(\begin{array}{cccc} 
1 & z & 0 & z^2 \\ z& 1 & z^2 & 0 \\ 0 & z^2 & 1 & z \\ z^2 & 0 & z & 1 \end{array} \right). \left(\begin{array}{c} g_1(z^3) \\ g_2(z^3)\\ g_3(z^3) \\ g_4(z^3) \end{array} \right) \, .
\end{equation*}
Dans la suite, on notera 
$$ \g(z) := \left(\begin{array}{c} g_1(z) \\ g_2(z)\\ g_3(z) \\ g_4(z) \end{array} \right) \;\; \text{ et } \;\; A(z) 
:= \left(\begin{array}{cccc} 1 & z & 0 & z^2 \\ z& 1 & z^2 & 0 \\ 0 & z^2 & 1 & z \\ z^2 & 0 & z & 1 \end{array} \right) \, .
$$
Par d\'efinition des suites ${\bf a}_i$, on a la relation affine $g_1(z) + g_2(z) + g_3(z) + g_4(z) = 1/(1-z)$. 
Nous allons montrer qu'il n'existe par contre pas de relation lin\'eaire homog\`ene non triviale entre les $g_i(z)$.  

\begin{lem}\label{lem: linind}
Les fonctions $g_1(z),g_2(z),g_3(z)$ et $g_4(z)$ sont lin\'eairement ind\'ependantes sur $\Q(z)$.
\end{lem}

\begin{proof}
Comme dans l'exemple pr\'ec\'edent, il s'agit d'une cons\'equence directe du th\'eor\`eme \ref{BaseRelationsLin\'eaires2}. 
Avec les notations du th\'eor\`eme \ref{BaseRelationsLin\'eaires2}, on a 
ici $n=4$, $d=2$, $q=3$ et $\nu=0$. On en d\'eduit que  $h=1$ et $c=(3^{28} \times 8- 3)/4$. 
Ainsi, l'existence d'une relation lin\'eaire non triviale entre les fonctions $g_i(z)$ est \'equivalente \`a la non nullit\'e de l'espace 
$\ker_{\mathbb Q} {\mathcal S}(1, (3^{28} \times 8- 3)/4,\g)$. 
La d\'efinition des fonctions $g_i(z)$ nous permet de calculer explicitement les huit premi\`eres colonnes de cette matrice. Il vient :  
\begin{equation*}
\left(\begin{array}{c}1 \\ 0 \\ 0 \\ 0 \\ 0 \\ 0 \\ 0 \\ 0  \end{array}\right) , 
\left(\begin{array}{c}0 \\ 1 \\ 0 \\ 0 \\1 \\ 0 \\ 0 \\ 0  \end{array}\right) , 
\left(\begin{array}{c}0 \\ 0\\1 \\ 0 \\0 \\ 1 \\ 0 \\ 0  \end{array}\right) ,  
\left(\begin{array}{c}1 \\ 0 \\ 0 \\ 0  \\0 \\ 0\\1 \\ 0    \end{array}\right) ,  
\left(\begin{array}{c}0 \\ 0\\0 \\ 1 \\1 \\ 0 \\ 0 \\ 0   \end{array}\right) ,  
\left(\begin{array}{c}0 \\ 0\\1 \\ 0 \\0 \\ 0\\0 \\ 1  \end{array}\right) ,  
\left(\begin{array}{c}0 \\ 0\\0 \\ 1 \\0 \\ 0 \\ 1 \\ 0  \end{array}\right) ,  
\left(\begin{array}{c}1 \\ 0\\0 \\ 0 \\0 \\ 0\\0 \\ 1  \end{array}\right).
\end{equation*}
On peut v\'erifier que ces huit vecteurs sont lin\'eairement ind\'ependants sur $\Q$. On en d\'eduit que  $\ker_{\mathbb Q} {\mathcal S}(1, (3^{28} \times 8- 3)/4,\g)=\{0\}$. 
Le th\'eor\`eme \ref{BaseRelationsLin\'eaires2} implique alors que les fonctions $g_1(z),g_2(z),g_3(z),g_4(z)$ sont lin\'eairement ind\'ependantes sur $\Q(z)$.
\end{proof}

Comme dans l'exemple pr\'ec\'edent, on voit facilement que le d\'eterminant de la matrice $A(z)$ a une unique racine 
dans le disque unit\'e ouvert qui est encore  $\phi := (1 - \sqrt{5})/2$. 
Les autres racines sont le nombre d'or et les deux racines cubiques primitives de l'unit\'e. 
L'ensemble des points singuliers de notre syst\`eme est donc \`a nouveau l'ensemble 
$$
{\mathcal E} := \left\{ \phi^{1/3^{l}} \mid l\geq 1\right\} \, .
$$
Le th\'eor\`eme \ref{thm: StructureRelationsLineaires} permet alors de montrer le r\'esultat suivant.

\begin{prop}
\label{resolutionzeros}
Soient $\alpha$, $0<\vert \alpha\vert <1$,  un nombre alg\'ebrique et  ${\boldsymbol \omega} := (\omega_1,\omega_2,\omega_3,\omega_4) \in \Q^4$ 
un vecteur de nombres alg\'ebriques non tous \'egaux. Alors, le nombre 
$$
\omega_1g_1(\alpha) + \omega_2g_2(\alpha) +\omega_3 g_3(\alpha)  + \omega_4 g_4(\alpha) 
$$
est alg\'ebrique si, et seulement si, il existe un entier $l\geq 1$ tel que $\alpha = \phi^{1/3^l}$ et 
$\boldsymbol \omega$ est de la forme ${\boldsymbol \mu} + \lambda(1,1,1,1)$ avec
${\boldsymbol \mu}\in \ker_{\Q} A_l(\alpha)$ et $\lambda\in\Q$. 
\end{prop}

\begin{proof}
Posons $\beta: = \omega_1g_1(\alpha) + \omega_2g_2(\alpha) +\omega_3 g_3(\alpha)  + \omega_4 g_4(\alpha) \in \Q$. 
En utilisant  la relation $g_1(z) + g_2(z) + g_3(z) + g_4(z) = 1/(1-z)$, 
il vient :  
$$
\left(\omega_1 - \beta(1-\alpha)\right) g_1(\alpha) + \cdots +\left(\omega_4 - \beta(1-\alpha)\right) g_4(\alpha)=0 \,.
$$
Par hypoth\`ese, les nombres $\omega_i - \beta(1-\alpha)$ ne sont pas tous nuls. 
D'autre part, les fonctions $g_1(z),g_2(z),g_3(z)$ et $g_4(z)$ sont lin\'eairement ind\'ependantes 
d'apr\`es le lemme \ref{lem: linind}. Donc, d'apr\`es le th\'eor\`eme \ref{thm: pphHomogene},  
le nombre $\alpha$ ne peut pas \^etre un point r\'egulier du syst\`eme. Il existe donc $l\geq 1$ tel que $\alpha^{3^l} = \phi$.  Le 
th\'eor\`eme \ref{thm: StructureRelationsLineaires} montre alors que le vecteur 
$$
{\boldsymbol \mu} := \left(\omega_1 - \beta(1-\alpha), \omega_2 - \beta(1-\alpha),\omega_3 - \beta(1-\alpha), \omega_4 - \beta(1-\alpha) \right) \, 
$$
appartient n\'ecessairement \`a $ \ker_{\Q} A_l(\alpha)$, 
ce qui permet de conclure. 
\end{proof}

\begin{rem}
Si l'on choisit $l=1$, c'est-\`a-dire si  $\alpha = \phi$, on obtient que $\ker_{\Q}A(\phi)$ est de dimension 1, engendr\'e par le vecteur $(1,1,-1,-1)$. 
On en d\'eduit facilement que les nombres $g_1(\phi),g_2(\phi),g_3(\phi)$ et $g_4(\phi)$ sont tous transcendants, bien que $\phi$ soit une singularit\'e du syst\`eme.  
Ce comportement est donc tr\`es diff\'erent du premier exemple. En effet, nous avons vu dans ce premier cas que les deux nombres $f_1(\alpha)$ et $f_2(\alpha)$ sont alg\'ebriques pour 
toute singularit\'e $\alpha$. 
\end{rem}



\bigskip

\noindent{\bf \itshape Remerciements.}\,\,---  Les auteurs remercient Patrice Philippon de leur avoir communiqu\'e une version pr\'eliminaire de 
la pr\'epublication \cite{PPH}, ainsi que Jean-Paul Allouche pour ses remarques concernant une premi\`ere version de ce texte.  
Ils remercient \'egalement Jason Bell pour leur avoir indiqu\'e 
 la monographie de Lang \cite{La} comme r\'ef\'erence possible pour le r\'esultat utilis\'e au d\'ebut de la d\'emonstration du lemme \ref{lem: reg}.



\appendix

\section{Conjecture de Cobham et nombres automatiques}\label{sec: cob}

La th\'eorie des $E$-fonctions, introduite par Siegel, doit sans conteste son succ\`es \`a la transcendence des nombres $e$ et $\pi$, ainsi qu'au   
th\'eor\`eme de Lindemann--Weierstrass.  On ne conna\^it en revanche aucune constante math\'ematique classique li\'ee \`a la valeur en un point alg\'ebrique 
d'une fonction mahl\'erienne. Cela explique sans doute le faible d\'eveloppement qu'a connu la th\'eorie de Mahler  
durant presque cinquante ans. Pourtant, d\`es 1968, Cobham jeta un pont entre la th\'eorie des automates finis et la transcendance, offrant  
ainsi \`a la th\'eorie de Mahler le probl\`eme n\'ecessaire \`a son d\'eveloppement. 
Nous rappelons bri\`evement ici l'histoire de ce probl\`eme qui a \'et\'e la source principale de motivation pour ce travail.    

\subsection{La suite de schiffres des nombres alg\'ebriques}

L'\'etude de la suite des chiffres de constantes math\'ematiques classiques comme 
$$
\begin{array}{cccc}
&\sqrt 2 &= &1.414\,213\, 562\, 373\, 095\, 048\, 801\, 688\, 724\, 209\, 698\, 078\, 569 \cdots \\ 
\text{ou} & &&\\
&  \pi& = &3.141\, 592\, 653\, 589\, 793\, 238\, 462\, 643\, 383\, 279\, 502\,884\, 197  \cdots 
\end{array}
$$
est source  de myst\`ere et de frustration depuis des d\'ecennies.   
Tandis que ces nombres admettent une desciption g\'eom\'etrique particuli\`erement simple, leurs 
suites de chiffres semblent au contraire \^etre le reflet de ph\'enom\`enes complexes. 
Plusieurs langages ont \'et\'e utilis\'es afin de formaliser ce constat : celui des probabilit\'es par \'E. Borel \cite{Bo09}, 
celui des syst\`emes dynamiques topologiques par Morse et Hedlund \cite{HM38} et enfin celui des machines de Turing et de la complexit\'e algorithmique par 
Hartmanis et Stearns \cite{HS}.  Chacun de ces points de vue conduit \`a son propre floril\`ege de conjectures, le plus souvent hors d'atteinte. 

On sait depuis Turing \cite{Turing} que les nombres r\'eels peuvent \^etre grossi\`erement divis\'es en deux cat\'egories. 
D'un c\^ot\'e, les nombres calculables,  dont le d\'eveloppement binaire ou d\'ecimal peut \^etre produit par une machine de  Turing, 
et de l'autre, les nombres incalculables, dont la complexit\'e \'echapera \`a jamais \`a la sagacit\'e des ordinateurs.  
Alors que la plupart des nombres ne sont pas calculables, les constantes math\'ematiques classiques, comme les nombres alg\'ebriques,  
le sont g\'en\'eralement.  En 1965, Hartmanis et Stearns \cite{HS} ont \'et\'e parmi les premiers \`a consid\'erer la question fondamentale 
de la difficult\'e du calcul d'un nombre r\'eel, introduisant les classes de complexit\'e en temps. 
La notion de complexit\'e en temps rend compte du nombre d'op\'erations \'el\'ementaires n\'ecessaires \`a une 
 machine de Turing d\'eterministe \`a plusieurs rubans pour produire les $n$ premiers chiffres du d\'eveloppement binaire d'un nombre donn\'e.   
Un nombre r\'eel est alors consid\'er\'e comme d'autant plus simple que ses chiffres  
peuvent \^etre calcul\'es rapidement par une machine de Turing. \`A la fin de leur article, Hartmanis et Stearns pose la question suivante : 

\medskip

{\it Existe-t-il des nombres alg\'ebriques irrationnels pour lesquels les $n$ premiers chiffres binaires peuvent \^etre calcul\'es en $O(n)$ 
op\'erations par une machine de Turing d\'eterministe \`a plusieurs rubans ?}

\medskip

Malheureusement, ce probl\`eme demeure encore largement ouvert (voir la discussion dans \cite{ACL}).

\subsection{Automates finis et m\'ethode de Mahler} 
En 1968, Cobham \cite{Co68c,Co68b,Cob68} r\'edige une s\'erie de rapports dans lesquels il propose de restreindre le probl\`eme de Hartmanis--Stearns 
\`a certaines classes de machines de Turing et en premier lieu aux automates finis. 
Dans \cite{Cob68}, Cobham \'enonce le \og th\'eor\`eme\fg{} suivant sans en donner de d\'emonstration. Le th\'eor\`eme \ref{thm: baker} 
en fournit finalement une, presque cinquante ans plus tard.

\medskip

\noindent {\bf\itshape \og Th\'eor\`eme\fg.} --- 
\emph{ Soient $f_1(z),\ldots,f_n(z)\in \mathbb Q\{z\}$  des fonctions solutions 
d'un syst\`eme du type (\ref{eq: systeme}) et analytiques sur le disque unit\'e ouvert. Soit $\alpha$, $0<\vert \alpha \vert <1$ un nombre rationnel. 
Alors pour tout $\lambda_1,\ldots,\lambda_n\in \mathbb Q$, le nombre 
$$
\lambda_1f_1(\alpha) + \cdots + \lambda_nf_n(\alpha)
$$ 
est soit rationnel, soit transcendant. 
}

\medskip
 
Cobham  montre dans \cite{Cob68} que si une suite $(a_n)_{n\geq 0}$ peut \^etre engendr\'ee par un automate fini, alors la s\'erie g\'en\'eratrice 
$f(z) := \sum_{n\geq 0}a_nz^n$ est mahl\'erienne. Cela jette un pont entre la m\'ethode de Mahler 
et la complexit\'e de la suite des chiffres des nombres alg\'ebriques. 
En particulier, cela montre que le \og th\'eor\`eme\fg{} implique le corollaire suivant.  Ce corollaire, devenu la conjecture de Cobham, 
fournira une motivation importante au d\'eveloppement de la m\'ethode de Mahler \`a partir des ann\'ees 80 et de la popularisation 
par Mend\`es France des travaux de Cobham aupr\`es 
des sp\'ecialistes de transcendance.  
 
 \medskip

\noindent {\bf\itshape  \og Corollaire\fg{} (conjecture de Cobham).} --- 
\emph{ Le d\'eveloppement dans une base enti\`ere d'un nombre alg\'ebrique irrationnel ne peut \^etre engendr\'e par un automate fini.
}

\medskip

\`A ce stade, il est int\'eressant de noter que Cobham ignorait vraisemblablement l'existence des travaux de Mahler. 
Il avait seulement connaissance de ceux de Siegel concernant les $E$-fonctions \`a travers le livre de Gel'fond \cite{Ge}.  
C'est donc de fa\c con ind\'ependante qu'il a red\'ecouvert les \'equation fonctionnelles mahl\'eriennes et c'est 
l'analogie avec la th\'eorie des $E$-fonctions qui l'a pouss\'e \`a conjecturer, avec une remarquable clairvoyance, l'alternative   
fondamentale donn\'ee dans son \og th\'eor\`eme\fg.
La conjecture de Cobham a finalement \'et\'e prouv\'ee dans \cite{AB07} par une approche totalement 
diff\'erente qui repose sur l'utilisation d'un outil diophantien puissant : une version $p$-adique du th\'eor\`eme du sous-espace (voir \cite{AB07,ABL}).  
Plus g\'en\'eralement, dans la direction du th\'eor\`eme, l'alternative $f(\alpha)$ est soit dans $\mathbb Q(\alpha)$, soit transcendant, a \'et\'e d\'emontr\'ee :
\begin{itemize}

\medskip
\item dans le cas o\`u $f(z)$ est une s\'erie automatique et $\alpha$  est l'inverse d'un nombre de Pisot ou 
de Salem par le premier auteur et Bugeaud \cite{AB07}.   

\medskip

\item dans le cas o\`u $f(z)$ est une s\'erie r\'eguli\`ere et $\alpha$ est l'inverse d'un entier par Bell, Bugeaud et Coons \cite{BBC}.   
\medskip

\end{itemize}
Ces r\'esultats n\'ecessitent \'egalement l'utilisation  du th\'eor\`eme du sous-espace $p$-adique et ne rel\`eve donc pas directement de la m\'ethode de Mahler. 

Dans les ann\'ees $80$, plusieurs auteurs, et en particulier Loxton et van der Poorten \cite{LvdP82,LvdP88,Lox88}, 
ont essay\'e de d\'emontrer le th\'eor\`eme \ref{thm: nishioka} de Nishioka. Ces derniers ont par ailleurs affirm\'e que la conjecture de Cobham en d\'ecoulerait. Or,  
il y a clairement deux obstructions majeures \`a une telle implication :   

\medskip

\begin{itemize}

\item[{\rm (i)}]   \'Etant  donn\'ee une s\'erie automatique transcendante $f(z)$,  on peut toujours trouver un syst\`eme mahl\'erien et un vecteur de solution 
$f_1(z):=f(z), f_2(z),\ldots,f_n(z)$.  Si $\alpha$ n'est pas une singularit\'e du syst\`eme,  la transcendance de $f(z)$ donne seulement, avec le th\'eor\`eme \ref{thm: nishioka}, 
qu'au moins l'un des nombres 
$f(\alpha),f_2(\alpha), \ldots,f_n(\alpha)$ est transcendant, \`a moins que les fonctions $f_1(z),\ldots,f_n(z)$ ne soient alg\'ebriquement ind\'ependantes.  
Mais il n'y a aucune raison {\it a priori} pour qu'une telle condition soit v\'erifi\'ee.

\medskip

\item[{\rm (ii)}]  M\^eme si l'on parvient \`a s'extraire du point pr\'ec\'edent, il se pourrait que le point $\alpha$ soit une singularit\'e du syst\`eme \'etudi\'e. 

\end{itemize}

\medskip

Le th\'eor\`eme \ref{thm: pph} de Philippon permet de surmonter le point (i) puisque l'on peut 
toujours s'assurer, quitte \`a r\'eduire le syst\`eme et y ajouter la fonction constante \'egale \`a $1$, que les fonctions $1,f_1(z):=f(z),\ldots,f_n(z)$ sont 
lin\'eairement ind\'ependantes.  
On en d\'eduit alors l'ind\'ependance lin\'eaire sur $\Q$ des nombres 
$1,f(\alpha),f_2(\alpha),\ldots,f_n(\alpha)$, ce qui donne  imm\'ediatement la transcendance de $f(\alpha)$, lorsque $\alpha$ n'est pas une singularit\'e du syst\`eme.  
Le fait que l'on puisse \'egalement s'affranchir du point (ii) en comprenant la nature de $f(\alpha)$ \'egalement aux points singuliers du syst\`eme est, en un sens, l'objet principal de cet article (voir les d\'emonstrations des th\'eor\`emes \ref{thm: baker} et 
\ref{thm: StructureRelationsLineaires}). 

\subsection{D\'eveloppements des nombres alg\'ebriques dans une base alg\'ebrique}

L'utilisation du th\'eor\`eme du sous-espace est plus souple que la m\'ethode de Mahler car elle ne requiert la pr\'esence d'aucune \'equation 
fonctionnelle. Par contre, lorsqu'elle s'applique, la m\'ethode de Mahler a plusieurs avantages.  
Elle donne lieu \`a des r\'esultats d'ind\'ependance alg\'ebrique et elle permet \'egalement de traiter le cas de toutes les bases alg\'ebriques et 
pas seulement celles qui sont des nombres de Pisot ou de Salem. 

Soit $\beta>1$ un nombre r\'eel non entier.  
 On d\'efinit l'application  $T_\beta$ sur $[0,1]$ par  
 $T_{\beta} : x\longmapsto \beta x \hbox{ mod } 1$. 
 Le $\beta$-d\'eveloppement d'un nombre $x\in [0,1[$, not\'e  $d_{\beta}(x)$, 
 est alors d\'efinit par :
$$
d_{\beta}(x) :=0.x_1x_2\cdots \, ,
$$
o\`u $x_i=\lfloor \beta T^{i-1}_{\beta}(x)\rfloor$. 
Ce d\'eveloppement co\"incide avec celui obtenu en utlisant l'algorithme glouton. 
Les chiffres $x_i$ appartiennent \`a l'ensemble  $\{0,1,\ldots, \lfloor\beta\rfloor\}$. 
Un exemple classique est le d\'eveloppement en base $\varphi$ (le nombre d'or). 
Dans ce cas pr\'ecis, la diff\'erence avec une base enti\`ere n'est pas flagrante, mais, de fa\c con g\'en\'erale, l'\'etude des bases alg\'ebriques 
est bien plus complexe que celle des bases enti\`eres.   
Le plus souvent, on ne sait par exemple m\^eme pas 
caract\'eriser les nombres ayant un $\beta$-d\'eveloppement ultimement p\'eriodique. 
Certaines bases, comme la base $3/2$, semble particuli\`erement retorse et nos connaissances sont alors 
tr\`es limit\'ees. Le corollaire \ref{cor: alt} donne directement le r\'esultat g\'en\'eral suivant. 

\begin{coro}
Soit $\beta>1$ un nombre alg\'ebrique r\'eel et $\alpha$ un nombre r\'eel alg\'ebrique n'appartenant pas \`a $\mathbb Q(\beta)$. 
Alors le $\beta$-d\'eveloppement de $\alpha$ ne peut \^etre engendr\'e par un automate fini.  
\end{coro}

Ainsi, le d\'eveloppement de $\sqrt 2$ en base $3/2$  ne peut \^etre engendr\'e par un automate fini.  Il semble difficile d'obtenir ce r\'esultat 
en utilisant l'approche de \cite{AB07}. Une raison technique pour cela est que la hauteur du nombre $3/2$ est \'egale \`a $3$ qui est strictement sup\'erieur \`a 
$3/2$.


\end{document}